\documentclass[final,leqno]{siamltex}
\usepackage{chngcntr}

\usepackage[arrow, matrix, curve]{xy}
\usepackage[latin1]{inputenc} 	
\usepackage[T1]{fontenc}
\usepackage[english]{babel}
\usepackage{lmodern}
\usepackage{amsmath}
\usepackage{upgreek}
\usepackage{amssymb}  
\usepackage{mathrsfs}
\usepackage{euscript}
\usepackage{enumerate}
\usepackage{xspace}
\usepackage{hyperref}  
\usepackage{framed}
\usepackage{cite}
\usepackage{fancyhdr}
\usepackage{bbm}
\usepackage{float}
\usepackage{comment}
\usepackage{graphicx}
\usepackage{textcomp}
\usepackage{url}
\usepackage{srcltx} 
\usepackage{hyperref} 
\hypersetup{draft=false, colorlinks=true, linkcolor=black, urlcolor=blue, citecolor=black}

\DeclareFontShape{OMX}{cmex}{m}{n}{
  <-7.5> cmex7
  <7.5-8.5> cmex8
  <8.5-9.5> cmex9
  <9.5-> cmex10
}{}
\SetSymbolFont{largesymbols}{normal}{OMX}{cmex}{m}{n}
\SetSymbolFont{largesymbols}{bold}  {OMX}{cmex}{m}{n}

\newtheorem{remark}[theorem]{Remark}

\newcommand{\laplace}{\Delta}

\newcommand{\usc}{u_{\mathrm{sc}}}
\newcommand{\ui}{u_{\mathrm{i}}}

\DeclareMathOperator{\kernel}{\mathcal{N}}

\DeclareMathOperator{\ran}{\mathcal{R}}

\DeclareMathOperator{\Real}{Re}

\DeclareMathOperator{\op}{op}

\DeclareMathOperator{\N}{\mathbb{N}}

\DeclareMathOperator{\R}{\mathbb{R}}
\DeclareMathOperator{\C}{\mathbb{C}}

\DeclareMathOperator{\DT}{\mathcal{D}}

\DeclareMathOperator{\Li}{\mathcal{L}}

\DeclareMathOperator{\Is}{\mathcal{L}_{\operatorname{is}}}
\DeclareMathOperator{\inv}{\operatorname{inv}}

\author{Heiko Hoffmann\thanks{Department of Mathematics, Siegen University, Walter-Flex-Str. 3, 57072 Siegen, Germany ({\tt heiko.hoffmann@uni-siegen.de})} \and Anne Wald\thanks{Department of Mathematics, Saarland University, PO Box 15 11 50, 66041 Saarbr\"ucken, Germany ({\tt anne.wald@num.uni-sb.de}).}}

\title{On parameter identification problems for elliptic boundary value problems in divergence form \\ Part I: An abstract framework}

\begin{document}

\maketitle

\begin{abstract}
 Parameter identification problems for partial differential equations are an important subclass of inverse problems. The parameter-to-state map, which maps the parameter of interest to the respective solution of the PDE or state of the system, plays the central role in the (usually nonlinear) forward operator.
Consequently, one is interested in well-definedness and further analytic properties such as continuity and differentiability of this operator w.r.t. the parameter in order to make sure that techniques from inverse problems theory may be successfully applied to solve the inverse problem. In this work, we present a general functional analytic framework suited for the study of a huge class of parameter identification problems including a variety of elliptic boundary value problems (in divergence form) with Dirichlet, Neumann, Robin or mixed boundary conditions. In particular, we show that the corresponding parameter-to-state operators fulfil, under suitable conditions, the tangential cone condition, which is often postulated for numerical solution techniques. This framework particularly covers the inverse medium problem and an inverse problem that arises in terahertz tomography.
\end{abstract}

\begin{keywords} 
inverse problems, parameter identification, inverse scattering, form methods, existence and uniqueness of weak solutions, Fr\'echet differentiability, tangential cone condition
\end{keywords}
\begin{AMS}
35J25; 46N40; 65J15; 65J22; 65N21; 78A46 
\end{AMS}

\section{Introduction and Motivation}
Many inverse problems that arise in the natural sciences are based on a physical model that is formulated as a partial differential equation, or rather a boundary or initial value problem. Applications are, for example, photoacoustic tomography (PAT) \cite{arridge16, wang09}, electrical impedance tomography (EIT) \cite{borcea02, uhlmann09}, ultrasound imaging \cite{ck1}, and various examples in nondestructive testing \cite{alifanov94, tanaka98}. \\
Inverse problems are commonly formulated using operator equations
\begin{displaymath}
 F(\theta) = g, \quad F: \mathcal{D}(F) \subseteq X \rightarrow Y,
\end{displaymath}
where $F$ is called the \emph{forward operator} and $X$ and $Y$ are suitable function spaces. 
In parameter identification the forward operator is expressed as the composition $F = Q \circ S$ of a parameter-to-state map $S$ and an observation operator $Q$. The operator $S$ maps the parameter of interest to the (weak) solution $u_{\theta} = S(\theta)$ of the respective boundary value problem, whereas the observation operator $Q$ describes the measuring process, i.e., the generation of the data $y = Q(u_{\theta})$ from the state $u_{\theta}$. \\
In this article, we address parameter-to-state operators, which often turn out to be nonlinear operators.
In general, the first step of a mathematical analysis of parameter identification problems is to show well-definedness as well as continuity and differentiability properties of the forward operator, particularly of the parameter-to-state map. The latter properties are required for many regularisation techniques that are used to find a stable solution of the usually ill-posed parameter identification problems. Examples are the classical Landweber method \cite{hns_cl}, Tikhonov regularisation \cite{ekn89}, Gauss-Newton methods \cite{bk97, qi00}, or sequential subspace optimisation techniques \cite{ws17, aw18}.
An overview of suitable techniques can be found in \cite{benning_burger_2018, ehn96, kns_itreg, skhk12}.

We derive a general framework that allows the treatment of a certain class of parameter-to-state operators that are linked to elliptic boundary value problems. To this end, we consider the variational formulation of the underlying boundary value problem, i.e., we are interested in weak solutions. In order to establish the well-definedness of the parameter-to-state operator, we have to show the existence and uniqueness of a solution of the respective variational problem. Similar framework, particularly suited for a wide class of time-dependent parameter identification problems, have been published in \cite{kaltenbacher17, akar16}. 

\vspace*{2ex}

The framework that is derived in this work is inspired by the analysis of the so-called \emph{scattering operator} as it occurs in inverse scattering problems such as the inverse medium problem, see, e.g., \cite{gbyc, gbpl,gbpl07}, and an inverse problem from terahertz (THz) tomography \cite{awts18}. In these examples, an object is illuminated by electromagnetic radiation $\ui$ at fixed frequencies $k_0 > 0$. The properties of the object, encoded in a material parameter $m$, lead to refraction, reflection and, in the case of THz tomography, absorption of the radiation $u$, which is the superposition $u= \ui + \usc$ of a given incident wave $\ui$ and the scattered wave $\usc$. The latter is the solution of the boundary value problem
\begin{align}
 \laplace \usc + k_0^2 (1-m) \usc &= k_0^2 m \ui && \text{in }\Omega, \label{hhg}\\ 
 \partial_{\nu} \usc - ik_0 \usc &= 0 && \text{on }\partial\Omega \label{rbc}
\end{align}
with Robin boundary conditions. The scattering operator is the parameter-to-state map $S: m \mapsto u := \ui + \usc$, i.e., it maps the material parameter $m$ to the resulting wave field $u$. More precisely, $\usc$ is the weak solution of this Helmholtz equation. Finally, the radiation is typically measured on a suitable curve around the object, determined by the domain $\Omega$. The inverse problem now consists in reconstructing $m$ from these measurements. Note that $m$ is real-valued in the inverse medium problem and complex-valued in THz tomography.\\
The respective variational problem is expressed, using a sesquilinear form $a$ and a functional $b$, via
\begin{displaymath}
 a(\usc,v) = b(v)
\end{displaymath}
for all suitable test functions $v$, and we are interested in a unique weak solution $\usc$. The Lax-Milgram lemma yields the desired result, if $a$ is a coercive and bounded sesquilinear form and $b$ is a bounded linear functional. However, this does not hold in general for the variational problems considered in the afore-mentioned context. 

\vspace*{2ex}

In this work, we are concerned with a more general framework, covering a wider range of boundary value problems resp. corresponding variational problems that arise from elliptic partial differential equations and include the scattering problems related to THz tomography or the inverse medium problem.
As we shall see, by using functional analytic tools such as a Riesz-type representation theorem and the Fredholm alternative, one can prove the existence of a unique weak solution, if the domain of the forward operator is defined on a set of certain admissible parameters.
 
Concerning the applications to elliptic boundary value problems in an upcoming paper,
we shall make use of the form methods introduced by Kato, see \cite{kato}, and Lions \cite{jll57}, which have been employed and hugely extended in various recent works by Arendt, ter Elst and others, see, e.g., \cite{at2012, at2012_2} and which have been applied in other relevant applications such as in \cite{watr09}. An overview of the functional analytic background, in particular in the complex-valued setting, can be found in \cite{sscs_en}.

The paper is organised as follows.
In the next section we specify the setting, i.e., we introduce the spaces that are involved as well as the properties of the considered forms. Within this general framework, we find, in Section \ref{operator theoretic formulation}, an operator theoretic reformulation of the problems we are interested in and prove, based on this, existence and uniqueness of a weak solution in Section \ref{section_ex_uni}. Following this, we illuminate the relation between our approach and the form methods mentioned above. Afterwards, we study the analytic properties of certain parameter-to-state operators in Section \ref{inverse problem}.
In the final section, we give a summary and outlook.

\section{Preliminaries} \label{section_prelim}
In this short section we fix the notation, collect some well-known facts, and introduce the abstract framework we shall work within.
\\
In what follows we consider vector spaces over $\mathbb K\in\{\R,\C\}$.
Let $(E,\tau_E)$ be a topological space, $(W,\|\cdot\|_W)$ a nontrivial reflexive Banach space
and $(V,\|\cdot\|_V)$, $(H,\|\cdot\|_H)$ and $(X,\|\cdot\|_X)$ Banach spaces.
We assume that $V\subseteq H$ with a continuous
 inclusion mapping and with embedding constant $\gamma>0$, i.e., the function $j:V\rightarrow H;\ v\mapsto v$ is continuous with
\begin{displaymath}\label{embedding constant}
\|j\|_{\text{op}}=\gamma.
\end{displaymath}
In particular, note that, although $V$ may also carry the relative topology induced by $H$, we assume throughout that $V$ is endowed with its own norm $\|\cdot\|_V$.
\\
Moreover, we denote by $W^*$ the space of antilinear functionals on $W$ and we endow it with the usual operator norm. 
\\
Furthermore, we consider a non-empty open subset $U\subseteq X$. Both $E$ and $U$ will later serve as definition sets for the parameters that shall be identified.  
\\ 
For normed spaces $(X_1,\|\cdot\|_{X_1})$, $(X_2,\|\cdot\|_{X_2})$, $(X_3,\|\cdot\|_{X_3})$ we denote by $\mathcal{S}(X_1\times X_2,X_3)$
the vector space of all continuous  sesquilinear (antilinear in the second argument) mappings $X_1\times X_2\rightarrow X_3$. Recall that
\begin{align*}
&\|\cdot\|_{\mathcal{S}(X_1\times X_2,X_3)}:\mathcal{S}(X_1\times X_2,X_3)\rightarrow\lbrack 0,\infty),\\
&\quad a\mapsto\sup\left\{\|a(x_1,x_2)\|_{X_3}:\, x_1\in X_1,\,x_2\in X_2\text{ with }\|x_1\|_{X_1},\|x_2\|_{X_2}\leq 1\right\}
\end{align*}
defines a norm on $\mathcal{S}(X_1\times X_2,X_3)$ and $(\mathcal{S}(X_1\times X_2,X_3),\|\cdot\|_{\mathcal{S}(X_1\times X_2,X_3)})$ is a Banach space, provided that $X_3$ is complete.
Note that elements of $\mathcal{S}(X_1\times X_2,X_3)$ are just bilinear in case of $\mathbb K=\R$.
For $a\in\mathcal{S}(X_1\times X_2,X_3)$ we define $a(x_1):=a(x_1,x_1)$.
\\
Moreover, $\Li(X_1,X_2)$ denotes the space of all bounded, linear mappings $X_1\to X_2$ and we endow this space with the usual operator norm denoted by
$\|\cdot\|_{\Li(X_1,X_2)}$ or simply $\|\cdot\|_{\text{op}}$, which turns $\Li(X_1,X_2)$ into a Banach space provided that $X_2$ is complete.
Instead of $\Li(X_1,X_1)$ we write $\Li(X_1)$ and we let $I_{X_1}$ denote the identity on $X_1$.
Furthermore, $X_1'$ denotes the topological dual space of $X_1$.
For the corresponding dual pairings we write
$\langle x_1,x_1'\rangle=x_1'(x_1)$, where $x_1\in X_1$, $x_1'\in X_1'$ or $x_1\in X^*$.
In addition, $\Is(X_1,X_2)$ denotes the set of all (topological) isomorphisms (i.e., linear homeomorphisms) between $X_1$ and $X_2$.
Recall that $\Is(X_1,X_2)$ is an open subset of $\Li(X_1,X_2)$, if $X_1$ and $X_2$ are Banach spaces.
In the case that $X_1=X_2$ we write $\Is(X_1)$ instead of $\Is(X_1,X_1)$.
If $\mathcal{H}$ is a Hilbert space, we denote the corresponding inner product by $(\cdot|\cdot)_{\mathcal{H}}$, where we drop the index, provided that no confusion is to be expected.

\vspace*{2ex}

For a subspace $\mathcal{D}\subseteq X_1$ and a linear mapping $A:\mathcal{D}\rightarrow X_2$, we denote by $\DT(A)$, $\ran(A)$ and $\kernel(A)$ the domain, the range and the null space, resp., and by $\|\cdot\|_A$ the corresponding graph norm. We say that $A$ is an operator from $X_1$ to $X_2$, even if $\mathcal{D}$ is a proper subspace of $X_1$.
If $\widetilde{\mathcal{D}}\subseteq X_1$ is another subspace and $\widetilde{A}:\widetilde{\mathcal{D}}\rightarrow X_2$ another linear operator, we write $A\subseteq \widetilde{A}$ provided that
$\mathcal{D}\subseteq\widetilde{\mathcal{D}}$ and $Ax=\widetilde{A}x$ for all $x\in\mathcal{D}$; in fact, we identify an operator $A:\mathcal{D}\rightarrow X_2$ with its graph $\{(x_1,Ax_1)|\,x_1\in\DT(A)\}$.
For an injective, linear mapping $A:\mathcal{D}\rightarrow X_2$ we put
$A^{-1}:=\{(Ax,x)|\,x\in\DT(A)\}$. In that case $A^{-1}$ is a univalent, linear operator with $\DT(A^{-1})=\ran(A)$.
\\
For $x_1\in X_1$ and $\varepsilon>0$ we set $B_\varepsilon(x_1):=\{u_1\in X_1:\,\|x_1-u_1\|\leq\varepsilon\}$.
\\
If $\Omega\subseteq X_1$ is non-empty and open and $f:\Omega\rightarrow X_2$ is Fr\'echet-differentiable at some point $x\in X_1$, we denote by
$\mathrm D_{\mathcal F}f(x)$ the Fr\'echet-derivative of $f$ at the point $x$.
\\
We consider continuous mappings
\begin{align*}
\mathfrak{a}_1:E\rightarrow\mathcal{S}(V\times W,\mathbb K),\, t\mapsto a_1^{(t)}:=\mathfrak{a}_1(t),
\end{align*}
\begin{align*}
\mathfrak{a}_2:U\rightarrow\mathcal{S}(H\times W,\mathbb K),\, m\mapsto a_2^{(m)}:=\mathfrak{a}_2(m),
\end{align*}
and
\begin{align*}
\mathfrak{c}:E\times U\rightarrow\mathcal{S}(H\times W,\mathbb{K});\,(t,m)\mapsto c^{(t,m)}:=\mathfrak{c}(t,m),
\end{align*}
where $E\times U$ carries the product topology,
and we assume that
\begin{align}\label{A1}
\sup_{\genfrac{}{}{0pt}{}{w\in W}{\|w\|_W=1}}|a_1^{(t)}(v,w)|\geq c(t)\|v\|_V
\end{align}
for all $v\in V$ and $t\in E$, where $c(t)>0$ for all $t\in E$ and $c(t)$ does not depend on $v$.
Moreover, we assume that for each $t\in E$ the sesquilinear form $\mathfrak{a}_1(t)$ is nondegenerate with respect to the second component, i.e., $a_1^{(t)}(v,w)=0$ for all $v\in V$ implies $w=0$. We notice that this forces $V\not=\{0\}$, as $W$ is nontrivial by assumption.
In particular, these assumptions are satisfied in the important case that $V=W$ (with equal norms) and $a_1^{(t)}$ is coercive, i.e.,
\begin{align}\label{coercivity}
\Real a_1^{(t)}(v,v)\geq c(t)\|v\|_V^2
\end{align}
for all $v\in V$ and some $c(t)>0$. Indeed, we then obtain 
\begin{displaymath}
\begin{split}
 \sup_{\genfrac{}{}{0pt}{}{w\in W}{\|w\|_W=1}}\big\lvert a_1^{(t)}(v,w)\big\rvert &=\sup_{\genfrac{}{}{0pt}{}{w\in V}{\|w\|_V=1}}\big\lvert a_1^{(t)}(v,w)\big\rvert \geq\|v\|_V\left|a_1^{(t)}\Big(\frac{1}{\|v\|_V}v,\frac{1}{\|v\|_V} v\Big)\right|\\
 &\geq \|v\|_V \Real\left(a_1^{(t)}\Big(\frac{1}{\|v\|_V}v,\frac{1}{\|v\|_V} v\Big)\right) \geq c(t)\|v\|_V
\end{split}
\end{displaymath}
for all $v\in V\setminus\{0\}$. Moreover, if $a_1^{(t)}(v,w)=0$ for all $v\in V$, then, in particular,
\begin{align*}
0=\Real a_1^{(t)}(w,w)\geq c(t)\|w\|_V^2,
\end{align*}
which yields $w=0$ since $c(t)>0$.

\vspace*{1ex}

For $t\in E$ and $m\in U$ let $C(t)$, $M(m)$ and $M(t,m)$ be positive real numbers satisfying
\begin{equation} \label{boundedness_a1}
 C(t)\geq\|\mathfrak{a}_1(t)\|_{\mathcal{S}(V\times W,\mathbb K)},
\end{equation}
\begin{equation} \label{boundedness_a2}
 M(m)\geq\|\mathfrak{a}_2(m)\|_{\mathcal{S}(H\times W,\mathbb K)},
\end{equation}
and
\begin{equation} \label{boundedness_c}
 M(t,m)\geq\|\mathfrak{c}(t,m)\|_{\mathcal{S}(H\times W,\mathbb K)}.
\end{equation}
Finally, let $\lambda:E\rightarrow\mathbb K$ be continuous.
We are especially interested in the case that
\begin{align*}
\mathfrak{c}(t,m)=\lambda(t)\mathfrak{a}_2(m)
\end{align*}
for all $t\in E$ and $m\in U$.
\vspace*{1ex}

Our first aim is to study, under various conditions, the existence and properties of solutions $u\in V$ to the problem
 \begin{align}\label{general 2}
 \forall\,w\in W:\,a_1^{(t)}(u,w)+c^{(t,m)}(u,w)=\varphi(w),
 \end{align}
 resp.
 \begin{align}\label{2}
 \forall\,w\in W:\,a_1^{(t)}(u,w)+\lambda(t) a_2^{(m)}(u,w)=\varphi(w),
 \end{align}
 where $\varphi\in W^*$ is given and $t\in E$ and $m\in U$ are parameters. 
 Problems \eqref{general 2} and \eqref{2} may be interpreted as the weak formulation of an elliptic boundary value problem, where $W$ serves as a space of test functions. In that case, the lower order terms of the corresponding differential operator are encoded in the form $c^{(t,m)}$ and they depend on the parameters $m$ and $t$, while $a_1^{(t)}$ essentially describes the highest order terms. 
 The solution space $V$ contains information on the boundary values. 
 
 In the inverse medium problem \cite{gbpl} or the inverse problem from THz tomography \cite{awts18} which we mentioned in the introduction, $m$ corresponds to a spatial material parameter, whereas $t$ represents the (fixed) frequency of the radiation. 

 \vspace*{2ex}
 
  An operator theoretic reformulation of our problem in the next section is the starting point of our studies.
 Afterwards, we will explore the dependence of the solution $u$ on $m$, $t$ and $\varphi$. In particular, we provide conditions guaranteeing that the dependence of $u$ on $m$ is continuously Fr\'echet-differentiable and the corresponding parameter-to-state operator satisfies the tangential cone condition, which indicates the quality of a local approximation of this operator by its linearisation. 
Finally, we sketch how to apply our abstract results to specific important examples.
More details will be delivered in a forthcoming paper.

\section{Operator theoretic formulation of \eqref{general 2}}\label{operator theoretic formulation}

\subsection{Associated operators}
In this subsection, we associate linear operators to the problem \eqref{general 2} in order to explore this problem using operator theoretic methods.
For that purpose,  we need the following two lemmas.
The first auxiliary result can be regarded as a Banach space version of the classical Lax-Milgram Lemma and it can be easily established applying the strategy used in the proof of Theorem 12 in \cite{Hayden}. For the reader's convenience we provide a complete proof.

\vspace*{1ex}

\begin{lemma}\label{Lax-Milgram}
For each $t\in E$ there exists an isomorphism $\mathcal T_t:V\rightarrow W^*$ such that
\begin{enumerate}
 \item $\|\mathcal T_t\|_{\Li(V,W^*)}\leq C(t)$, 
 \item $\|\mathcal T_t^{-1}\|_{\Li(W^*,V)}\leq\frac{1}{c(t)}$, and 
 \item $a_1^{(t)}(v,w)=(\mathcal T_t v)\lbrack w\rbrack$ for all $v\in V$ and $w\in W$.
\end{enumerate}
\end{lemma}

\vspace*{1ex}

\begin{proof}
We claim that
\begin{align*}
\mathcal T_t : V\rightarrow W^*, \ v\mapsto a_1^{(t)}(v,\,\cdot\,)
\end{align*}
has the desired properties. Clearly, $\mathcal T_t$ is well-defined and linear with
\begin{align*}
\|\mathcal T_t v \|_{W^*} = \sup_{\genfrac{}{}{0pt}{}{w\in W}{\|w\|_W\leq1}} \big\lvert a_1^{(t)}(v,w) \big\rvert \leq C(t)\|v\|_V.
\end{align*}
This inequality further implies $\mathcal T_t\in\Li(V,W^*)$ with $\|\mathcal T_t\|_{\Li(V,W^*)}\leq C(t)$. Moreover, we estimate
\begin{align*}
\inf_{\genfrac{}{}{0pt}{}{v\in V}{\|v\|_V=1}}\|\mathcal{T}_t v\|_{W^*}
=\inf_{\genfrac{}{}{0pt}{}{v\in V}{\|v\|_V=1}}\sup_{\genfrac{}{}{0pt}{}{w\in W}{\|w\|_W=1}} \big\lvert a_1^{(t)}(v,w) \big\rvert \geq c(t),
\end{align*}
using \eqref{A1}.
This shows that $\mathcal T_t$ is injective and that the inverse
\begin{align*}
\mathcal{T}_t^{-1}:\mathcal{T}_t(V)\rightarrow V, \ \mathcal{T}_tv\mapsto v
\end{align*}
is bounded with $\|\mathcal T_t^{-1}\|_{\Li(\mathcal{T}_t(V),V)}\leq\frac{1}{c(t)}$, where $\mathcal{T}_t(V)$ is endowed with the restriction of the norm $\|\cdot\|_{W^*}$. In particular, $\mathcal{T}_t(V)$ and $V$ are topologically isomorphic. Hence,
$\mathcal{T}_t(V)$ is a Banach space, too, thus a closed subspace of $W^*$. So, it remains to verify that $\mathcal{T}_t(V)=W^*$. Suppose to the contrary that this fails.
By the Hahn-Banach theorem and the closedness of $\mathcal{T}_t(V)$, we can find a $\chi\in(W^*)'\setminus\{0\}$ such that $\chi|_{\mathcal{T}_t(V)}=0$. We consider 
\begin{align*}
\widehat\chi:W'\rightarrow\mathbb K, \ \psi\mapsto\overline{\chi(\overline{\psi})},
\end{align*}
which is an element of the bidual space of $W$, where
\begin{align*}
\overline{\psi}:W\rightarrow\mathbb K;\, w\mapsto\overline{\psi(w)}
\end{align*}
and the bar denotes complex conjugation. Since $W$ is reflexive, there exists a
$w\in W$ such that
$\widehat{\chi}(\psi)=\psi(w)$ for all $\psi\in W'$.
This yields $\overline{\chi(\varphi)}=\overline{\varphi}(w)$ resp. $\chi(\varphi)=\varphi(w)$ for all $\varphi\in W^*$. As a consequence, we derive on the one hand $w\not=0$, as $\chi$ is nontrivial, and on the other hand
\begin{align*}
 \big|a_1^{(t)}(v,w)\big| = |(\mathcal T_tv)\lbrack w\rbrack|  
=|\chi(\mathcal T_tv)|=0
\end{align*}
for all $v\in V$, which implies $w=0$ because $a_1^{(t)}$ is nondegenerate w.r.t. the second argument, which contradicts our assumption. 
\end{proof}

\vspace*{1ex}

\begin{remark}
Theorem 1.1 in \cite{KozonoYanagisawa} is another Banach space version of the classical Lax-Milgram lemma as our Lemma \ref{Lax-Milgram}. Note, however, that none of these two results completely implies the respective other one.
\end{remark}

\vspace*{1ex}

The next lemma constitutes an important step towards the possibility of using operator theory in treating problem \eqref{general 2}.

\vspace*{1ex}

\begin{lemma}\label{general Ausgangslemma}
For each pair $(t,m)\in E\times U$ there exists a unique bounded operator
$\mathcal C_{t,m}:H\rightarrow H$ with $\mathcal C_{t,m}(H)\subseteq V$ and with
 \begin{align}\label{general 1}
 a_1^{(t)}(\mathcal C_{t,m}x,w)=c^{(t,m)}(x,w)
 \end{align}
 for every $x\in H$ and each $w\in W$.
 In addition, the following assertions are valid.
\begin{enumerate}
\item The mapping $\mathcal C:E\times U\rightarrow\Li(H), \ (t,m)\mapsto\mathcal C_{t,m}$ is continuous.
\item The part of $\mathcal C_{t,m}$ in $V$, i.e., the linear operator
 \begin{align*}
 \mathcal C_{t,m}^V:V\rightarrow V, \ v\mapsto\mathcal C_{t,m}v
 \end{align*}
 is bounded and the mapping $\mathcal C^V:E\times U\rightarrow\Li(V), \ (t,m)\mapsto\mathcal C_{t,m}^V$ is continuous.
\item We have $\|\mathcal C_{t,m}x\|_V\leq\frac{M(t,m)}{c(t)}\cdot\|x\|_H$ for each $x\in H$.
\item The operators $\mathcal C_{t,m}$ and $\mathcal{C}_{t,m}^V$ are both compact if the embedding $j:V\rightarrow H$ is compact.
\end{enumerate}
\end{lemma}

\vspace*{1ex}

\begin{proof}
Let $t\in E$ and $w\in W$. Thanks to Lemma \ref{Lax-Milgram} we have an isomorphism
\begin{align*}
\mathcal{T}_t:V\rightarrow W^*
\end{align*}
with
\begin{displaymath}
 \|\mathcal{T}_t\|_{\Li(V,W^*)}\leq\|\mathfrak{a}_1(t)\|_{\mathcal S(V\times W,\mathbb K)}, \quad \|\mathcal{T}_t^{-1}\|_{\Li(W^*,V)}\leq\frac{1}{c(t)},
\end{displaymath}
and
\begin{align}\label{4}
a_1^{(t)}(v,w)=(\mathcal T_t v)\lbrack w\rbrack
\end{align}
for all $v\in V$. One easily verifies that the continuity of $\mathfrak{a}_1$ implies that the function 
\begin{align*}
\mathcal T:E\rightarrow\Li(V,W^*), \ t\mapsto\mathcal T_t
\end{align*}
is continuous. 
For $t\in E$ and $m\in U$ we consider the mapping
\begin{align*}
\mathcal B_{t,m}:H\rightarrow W^*, \ x\mapsto c^{(t,m)}(x,\,\cdot\,).
\end{align*}
We first observe that $\mathcal B_{t,m}$ is well-defined. Indeed, for $x\in H$ the mapping $c^{(t,m)}(x,\,\cdot\,)$ is clearly antilinear.
We further obtain
\begin{align*}
\big\lvert c^{(t,m)}(x,w) \big\rvert \leq \big\lVert c^{(t,m)}\big\rVert_{\mathcal{S}(H\times W,\mathbb K)}\cdot\|x\|_H\cdot\|w\|_W.
\end{align*}
Hence, $c^{(t,m)}(x,\,\cdot\,)$ is continuous with
$\|c^{(t,m)}(x,\,\cdot\,)\|_{W^*}\leq  \|c^{(t,m)}\|_{\mathcal{S}(H\times W,\mathbb K)}\cdot\|x\|_H$.
Since $\mathcal B_{t,m}$ is linear, as one easily verifies, the last inequality also shows that $\mathcal B_{t,m}$ is bounded
with $\|\mathcal B_{t,m}\|_{\Li(H,W^*)}\leq \|c^{(t,m)}\|_{\mathcal{S}(H\times W,\mathbb K)}$.
Moreover, we claim that the mapping
\begin{align*}
\mathcal B:E\times U\rightarrow\Li(H,W^*), \ m\mapsto\mathcal{B}_{t,m}
\end{align*}
is continuous. In fact, for $t,\widetilde t\in E$ and $m,\widetilde m\in U$ we compute
\begin{align*}
\|\mathcal B_{t,m}-\mathcal B_{\widetilde t,\widetilde m}\|_{\Li(H,W^*)}
&=\sup_{\genfrac{}{}{0pt}{}{x\in H}{\|x\|_H\leq 1}}\|c^{(t,m)}(x,\,\cdot\,)-c^{(\widetilde t,\widetilde{m})}(x,\,\cdot\,)\|_{W^*}\\
&=\sup_{\genfrac{}{}{0pt}{}{x\in H}{\|x\|_H\leq 1}}\sup_{\genfrac{}{}{0pt}{}{w\in W}{\|w\|_W\leq 1}}
|c^{(t,m)}(x,w)-c^{(\widetilde t,\widetilde m)}(x,w)|\\
&=\|c^{(t,m)}-c^{(\widetilde t,\widetilde m)}\|_{\mathcal S(H\times W,\mathbb K)}\\
&=\|\mathfrak{c}(t,m)-\mathfrak{c}(\widetilde t,\widetilde m)\|_{\mathcal S(H\times W,\mathbb K)}\xrightarrow[(t,m)\to(\widetilde t,\widetilde m)]{}0.
\end{align*}
Recall that we consider the canonical embedding
$j:V\rightarrow H, \ v\mapsto v$
(with embedding constant $\gamma$, see \eqref{embedding constant}).
We  put
\begin{align*}
\widetilde{\mathcal C}_{t,m}:=\mathcal T_t^{-1}\mathcal B_{t,m}\in\Li(H,V)
\end{align*}
as well as
\begin{align*}
{\mathcal C}_{t,m}:=j\widetilde{\mathcal C}_{t,m}\in\Li(H)
\end{align*}
and we consider
\begin{align*}
\mathcal{C}:E\times U\rightarrow\Li(H);\,(t,m)\mapsto\mathcal{C}_{t,m}.
\end{align*}
If the inclusion map $j$ is compact, $\mathcal C_{t,m}$ is compact as a product of a compact and a bounded linear operator.
Furthermore, $\mathcal C_{t,m}(H)\subseteq V$.
Thus, we immediately see that
$\mathcal{C}_{t,m}^V=\mathcal{T}_t^{-1}\mathcal{B}_{t,m}j$. Hence, $\mathcal{C}_{t,m}^V$ is a bounded operator and it is compact as the product of a bounded and a compact operator provided that $j$ is compact.
One easily verifies that the mapping
\begin{align*}
\Psi:\Li(W^*,V)\times\Li(H,W^*)\rightarrow\Li(H), \ (F,G)\mapsto jFG
\end{align*}
is a continuous bilinear mapping (with norm bounded by $\gamma$). Moreover, the mapping
\begin{align*}
\inv_{V,W^*}:\Is(V,W^*)\rightarrow\Is(W^*,V), \ T\mapsto T^{-1}
\end{align*}
is continuous.
Therefore $f:=\inv_{V,W^*}\circ\mathcal T$
and thus
\begin{align*}
g:E\times U\rightarrow\Li(W^*,V)\times\Li(H,W^*), \ (t,m)\mapsto (f(t),\mathcal B_{t,m})
\end{align*}
are continuous, too. Hence, $\mathcal C=\Psi\circ g$ is continuous.
Analogously, one can show that $\mathcal{C}^V$ is continuous.

For every $x\in H$ and $w\in W$ we estimate (see also above) 
\begin{align*}
\|\mathcal C_{t,m}x\|_V
=&\|\widetilde{\mathcal C}_{t,m}x\|_V\leq\|\mathcal T_t^{-1}\|_{\Li(W^*,V)}\cdot\|\mathcal B_{t,m}x\|_{W^*}
\leq\frac{1}{c(t)}\cdot\|\mathcal B_{t,m}x\|_{W^*}\\
\leq &\frac{1}{c(t)}\cdot\|\mathcal B_{t,m}\|_{\Li(H,W^*)}\cdot\|x\|_H
\leq\frac{1}{c(t)}\|c^{(t,m)}\|_{\mathcal{S}(H\times W,\mathbb K)}\cdot\|x\|_H\\
\leq &\frac{M(t,m)}{c(t)}\cdot\|x\|_H
\end{align*}
and we compute
\begin{align*}
a_1^{(t)}(\mathcal C_{t,m}x,w)
&=a_1^{(t)}(\widetilde{\mathcal C}_{t,m}x,w)
=a_1^{(t)}(\mathcal T_t^{-1}\mathcal B_{t,m} x,w)
\stackrel{\eqref{4}}{=}(\mathcal T_t\mathcal T_t^{-1}\mathcal B_{t,m} x)\lbrack w\rbrack\\
&=B_{t,m}x\lbrack w\rbrack=c^{(t,m)}(x,w).
\end{align*}
Consequently, $\mathcal C$  and $\mathcal C_{t,m}$ are mappings of the desired type
and assertion a) -- d) are established.

In order to finish the proof, it only remains to show that $\mathcal{C}_{t,m}$ is unique. For this purpose
let $\mathcal C_{t,m}'\in\Li(H)$ be another operator with $\mathcal C'(H)\subseteq V$ and
 \begin{align*}
 a_1^{(t)}(\mathcal C_{t,m}'x,w)=c^{(t,m)}(x,w)
 \end{align*}
 for every $x\in H$ and each $w\in W$, where $t\in E$ and $m\in U$. This yields 
 \begin{align*}
 \mathcal T_t\mathcal C_{t,m}'x\stackrel{\eqref{4}}{=}a_1^{(t)}(\mathcal{C}_{t,m}'x,\,\cdot\,)
 =c^{(t,m)}(x,\,\cdot\,)=\mathcal B_{t,m}x,
 \end{align*}
 which implies
 \begin{align*}
 \mathcal C_{t,m}'x=\mathcal T_t^{-1}\mathcal B_{t,m}x=\widetilde{\mathcal C}_{t,m}x=\mathcal C_{t,m}x.
 \end{align*}
 As a result, we have shown that $\mathcal C_{t,m}$ is unique.
\end{proof}

\vspace*{1ex}

\begin{definition}
For $(t,m)\in E\times U$ we call problem \eqref{general 2} \emph{strongly well-posed} if and only if for each $\varphi\in W^*$ there exists precisely one $u\in V$
such that \eqref{general 2} is satisfied.
\end{definition}

\vspace*{1ex}

\begin{proposition}\label{well-posedness}
Let $(t,m)\in E\times U$.
\begin{enumerate}
\item For fixed $\varphi\in W^*$, $u\in V$ solves \eqref{general 2} if and only if $\mathcal T_t(I_V+\mathcal{C}_{t,m}^V)u=\varphi$.
\item The subsequent statements are equivalent.
 \begin{enumerate}
  \item Problem \eqref{general 2} is strongly well-posed.
  \item The operator $I_V+\mathcal{C}_{t,m}^V$ is bijective.
 \end{enumerate}
 In that case, the operator $I_V+\mathcal{C}_{t,m}^V$ possesses a bounded inverse and the unique solution to problem \eqref{general 2} depends continuously on the data $\varphi$.
 \item If the embedding $j$ is compact and the condition
 \begin{align}\label{general injectivity assumption 1}
\bigg(\forall\,w\in W:\,a_1^{(t)}(u,w)+c^{(t,m)}(u,w)=0\bigg) \ \Longrightarrow \ u=0
 \end{align}
 is satisfied, then problem \eqref{general 2} is strongly well-posed.
\end{enumerate}
\end{proposition}

\vspace*{1ex}

\begin{proof}
\emph{on a)}: For all $w\in W$, we compute, using \eqref{4},
{\allowdisplaybreaks
\begin{align*}
\left\langle w,\mathcal T_t(I_V+\mathcal{C}_{t,m}^V)u\right\rangle
&=a_1^{(t)}((I_V+\mathcal{C}_{t,m}^V)u,w)\\
&=a_1^{(t)}(u,w)+a_1^{(t)}(\mathcal{C}_{t,m}u,w)\\
&=a_1^{(t)}(u,w)+c^{(t,m)}(u,w),
\end{align*}
}
which implies the assertion.
\\[0,2cm]
\emph{on b)}: Since $\mathcal{T}_t$ is an isomorphism, the stated equivalence follows immediately from part a).
So, in the case that  $I_V+\mathcal{C}_{t,m}^V$ is
bijective, it possesses a bounded inverse due to the open mapping theorem.
Moreover, in this situation the unique solution $u$ to problem \eqref{general 2} is given by $u=(I_V+\mathcal{C}_{t,m}^V)^{-1}\mathcal{T}_t^{-1}\varphi$ and, consequently, depends
continuously on the given $\varphi$.
\\[0,2cm]
\emph{on c)}: Assume that $j$ is compact and condition \eqref{general injectivity assumption 1} is met.
By part a), condition \eqref{general injectivity assumption 1} is equivalent to $\kernel(\mathcal{T}_t(I_V+\mathcal{C}_{t,m}^V))=\{0\}$.
So, $I_V+\mathcal{C}_{t,m}^V$ is injective.
By Lemma \ref{general Ausgangslemma}, we derive that $\mathcal{C}_{t,m}^V$ is compact.
 Hence, $I_V+\mathcal{C}_{t,m}^V$ is an isomorphism by the Fredholm alternative (see, e.g., Theorem 15.9 in \cite{Fabian_et_al}). The assertion follows from part b).
\end{proof}

\vspace*{1ex}
 
 An important special case, in particular within a Hilbert space setting, occurs if $V$ coincides with $W$ and $V$ is densely embedded into $H$. Moreover, in that case a more detailed analysis of the involved operators is accessible.
Hence, for the remainder of this subsection we assume that $V=W$ and that $j$ has dense range, i.e., $V$ is dense in $H$.
We especially emphasise that $V$ is reflexive and non-trivial.
For $t\in E$ and $m\in U$ we put $a_{t,m}:=a:=a_1^{(t)}+c^{(t,m)}$ and we define
\begin{align*}
A_{t,m}:=A:=\left\{(u,\varphi)\in V\times H^*\Big|\,\forall\,v\in V:\,a_{t,m}(u,v)=\varphi(v)\right\}
\end{align*}
and
\begin{align*}
A_1^{(t)}:=\left\{(u,\varphi)\in V\times H^*\Big|\,\forall\,v\in V:\,a_1^{(t)}(u,v)=\varphi(v)\right\}.
\end{align*}
Using that $j$ has dense range, it is easy to show that $A_{t,m}$ and $A_1^{(t)}$ are univalent, linear relations, i.e., linear operators.
For given $\varphi\in H^*$ problem \eqref{general 2} may now be reformulated as follows: find $u\in\DT(A_{t,m})$ such hat
\begin{align*}
A_{t,m}u=\varphi.
\end{align*}
The operator $A_{t,m}$ corresponds to the differential operator governing the boundary value problems in the weak formulation.

\vspace*{1ex}

\begin{definition}
We still assume at this point that $V=W$ and that $j$ has dense range.
In that case, we call problem \eqref{general 2} \emph{$H$-well-posed} if for each $\varphi\in H^*$ there exists precisely one $u\in V$ such that
 \begin{align}\label{general 2 H-version}
 \forall\,v\in V:\,a_1^{(t)}(u,v)+c^{(t,m)}(u,v)=\varphi(v).
 \end{align}
\end{definition}

\vspace*{1ex}

By the very definitions, it is clear that problem \eqref{general 2} is $H$-well-posed if and only if $A_{t,m}$ is bijective.

Let $\varphi\in H^*$. Since the embedding $j$ is continuous, we have $\varphi\circ j\in V^*$, i.e., the mapping
 \begin{align*}
 j^\star:H^*\rightarrow V^*;\,\psi\mapsto\psi\circ j.
 \end{align*}
 is well-defined and, moreover, linear and bounded.
 Furthermore, it is injective (see Theorem \ref{general factorisation theorem} below).
Thus, if \eqref{general 2} is strongly well-posed or, equivalently, $I_V+\mathcal{C}_{t,m}^V$ is bijective, then problem  \eqref{general 2} is apparently $H$-well-posed, too.
However, the converse may fail in general because one can think, thanks to $j^\star$, of $H^*$ as a proper subspace of $V^*$ 
so that $H$-well-posedness is a weaker condition than being strongly well-posed: there are simply less conditional equations to be satisfied in order to guarantee $H$-well-posedness.

The next theorem gives a detailed analysis of the operators $A_{t,m}$ and $A_1^{(t)}$ and of the relationships among them 
as well as to $\mathcal{C}^V_{t,m}$. 
\vspace*{1ex}

\begin{theorem}\label{general factorisation theorem}
We consider $j^\star:H^*\rightarrow V^*;\,\psi\mapsto\psi\circ j$.
For $(t,m)\in E\times U$ the following assertions are valid.
\begin{enumerate}
 \item $A_{t,m}$ is a closed operator from $H$ to $H^*$.
 \item $j^\star$ is injective with dense range and $\|j^\star\|_{\op}=\gamma$.
 \item $A_1^{(t)}=(j^\star)^{-1}\mathcal{T}_t$; in particular, $A_1^{(t)}$ is a densely defined, continuously invertible, closed operator from $H$ to $H^*$.
 \item $\kernel(A_{t,m})=\kernel(I_V+\mathcal{C}_{t,m}^{V})$ and
 $\ran(A_{t,m})=(j^\star)^{-1}\left(\ran(\mathcal{T}_t(I_V+\mathcal{C}_{t,m}^V))\cap\ran(j^\star)\right)$.
 \item $A=A_{t,m}=A_1^{(t)}(I_V+\mathcal{C}_{t,m}^V)$.
 \item The subsequent statements are equivalent.
 \begin{enumerate}
  \item Problem  \eqref{general 2} is $H$-well-posed.
  \item The operator $A_{t,m}$ is bijective.
  \item The operator $I_V+\mathcal{C}_{t,m}^{V}$ is injective with $\ran(j^\star)\subseteq\ran(\mathcal{T}_t(I_V+\mathcal{C}_{t,m}^{V}))$.
 \end{enumerate}
 In that case, $A_{t,m}$ has a bounded inverse.
  \item Assume that  \eqref{general 2} is $H$-well-posed. Then the mapping
\begin{align*}
\mathcal{J}:\DT(A_{t,m})\rightarrow\DT(A_1^{(t)});\,u\mapsto (I_V+\mathcal{C}_{t,m}^V)u
\end{align*} 
is well-defined and bijective. Furthermore, $\mathcal{J}$ is continuous if both spaces $\DT(A)$ and $\DT(A_1^{(t)})$ are endowed with the respective graph norms where we consider $A_{t,m}$ and $A_1^{(t)}$ as operators from $H$ to $H^*$. In particular, $\mathcal{J}$ is an isomorphism.
 \item The operator $j^\star A_{1}^{(t)}$ is closable as an operator from $V$ to $V^*$
with $\overline{j^\star A_{1}^{(t)}}= \mathcal{T}_t$.
Suppose additionally that problem \eqref{general 2} is strongly well-posed. Then the operator $j^\star A_{t,m}$ is also closable as an operator from $V$ to $V^*$
with
\begin{align*}
\overline{j^\star A_{t,m}}=\mathcal{T}_t(I_V+\mathcal{C}_{t,m}^V).
\end{align*}
\end{enumerate}
\end{theorem}

\vspace*{1ex}

\begin{proof}
\emph{on a)}: Take an arbitrary sequence $(u_n,\varphi_n)_n$ in $A_{t,m}$ converging in $H\times H^*$ to $(u,\varphi)$.
In particular, $(u_n)_n$ converges in $H$ weakly to $u$.
Furthermore, we recall that $u_n\in V$ for all $n\in\N$.
Using \eqref{A1},
pick a $v_n\in V$ for each $n\in\N$ such that $\|v_n\|_V=1$ and
\begin{align}\label{finding limit point}
|a_1^{(t)}(u_n,v_n)|\geq\frac{c(t)}{2}\|u_n\|_V.
\end{align}
As $V$ is reflexive by assumption, we may extract a subsequence $(v_{n_k})_k$ weakly converging to a $v\in V$ due to the Banach-Alaoglu theorem.
One immediately sees that $\lim_{k\to\infty}\varphi_{n_k}(v_{n_k})=\varphi(v)$.
Since $\lim_{k\to\infty}u_{n_k}=u$ in $H$, we obtain
$$\lim_{k\to\infty}c^{(t,m)}(u_{n_k},\cdot)=c^{(t,m)}(u,\cdot)$$ with convergence in $H^*$.
Therefore, the same considerations as before yield
$$\lim_{k\to\infty}c^{(t,m)}(u_{n_k},v_{n_k})=c^{(t,m)}(u,v),$$
too.
Hence,
\begin{align*}
a_1^{(t)}(u_{n_k},v_{n_k})=\varphi_{n_k}(v_{n_k})-c^{(t,m)}(u_{n_k},v_{n_k})\xrightarrow[k\to\infty]{}\varphi(v)-c^{(t,m)}(u,v).
\end{align*}
As a result, the sequence $(a_1^{(t)}(u_{n_k},v_{n_k}))_k$ is bounded. Consequently, thanks to \eqref{finding limit point}, the sequence $(u_{n_k})_k$
is bounded in $V$. 
Employing once again the Banach-Alaoglu theorem, we assume w.l.o.g. that $(u_{n_k})_k$ converges in $V$ weakly to some $u_0\in V$. 
Then $(u_{n_k})_k$ also converges in $H$ weakly to $u_0$ because the embedding $j$ is continuous. The uniqueness of weak limits
implies $u=u_0$ and thus $u\in V$. Now, it is clear that $\lim_{n\to\infty}a_{t,m}(u_n,w)=a_{t,m}(u,w)$ and $\lim_{n\to\infty}\varphi_n(w)=\varphi(w)$ for all $w\in V$.
From this we conclude $(u,\varphi)\in A_{t,m}$. 
\\[0,2cm]
\emph{on b)}: This is essentially a standard result from functional analysis and follows directly from the facts that $V$ is reflexive and that $j$ is injective with dense range and with $\|j\|_{\op}=\gamma$. 
\\[0,2cm]
\emph{on c)}:
Let $u\in\DT((j^\star)^{-1}\mathcal{T}_t)$, i.e., $u\in V$ with $\mathcal{T}_tu\in\DT((j^\star)^{-1})=\ran(j^\star)$.
Consequently, there exists $\varphi\in H^*$ such that $\mathcal{T}_tu=\varphi j$.
We therefore calculate
\begin{align*}
\varphi(v)=\left\langle v,\varphi j\right\rangle=\left\langle v,\mathcal{T}_tu\right\rangle=a_1^{(t)}(u,v)
\end{align*}
for all $v\in V$
which shows $(u,\varphi)\in A_1^{(t)}$. This means $u\in\DT(A_1^{(t)})$ and $A_1^{(t)}u=\varphi=(j^\star)^{-1}\mathcal{T}_tu$.
It only remains to verify that $\DT(A_1^{(t)})\subseteq\DT((j^\star)^{-1}\mathcal{T}_t)$ in order to show that
$A_1^{(t)}=(j^\star)^{-1}\mathcal{T}_t$.
Let $u\in\DT(A_1^{(t)})$. Then,
\begin{align*}
\left\langle v,\mathcal{T}_tu\right\rangle=a_1^{(t)}(u,v)=\langle v,A_1^{(t)}u\rangle=\langle jv,A_1^{(t)}u\rangle=\langle v,j^\star A_1^{(t)}u\rangle
\end{align*}
for all $v\in V$ and thus $\mathcal{T}_tu=j^\star A_1^{(t)}u$, i.e., $\mathcal{T}_tu\in\ran(j^\star)$ and $u\in\DT((j^\star)^{-1}\mathcal{T}_t)$.
By part b), $(j^\star)^{-1}$ is continuously invertible by $j^\star$
and densely defined.
Hence, $A_1^{(t)}$ possesses a bounded inverse given by $\mathcal{T}_t^{-1}j^\star$.
Consequently, $A_1^{(t)}$ is closed and $\DT(A_1^{(t)})=\mathcal{T}_t^{-1}(\ran(j^\star))$ is dense in $V$ and thus in $H$.
\\[0,2cm]
\emph{on d)}:
This is a direct consequence of part c) and e).
\\[0,2cm]
\emph{on e)}:
Let $(u,\varphi)\in A_{t,m}$. Then $u\in V$ and $a_1^{(t)}(u,v)+c^{(t,m)}(u,v)=a(u,v)=\varphi(v)$ for all $v\in V$.
Due to Lemma \ref{general Ausgangslemma}, this yields
\begin{align*}
a_1^{(t)}((I_V+\mathcal{C}_{t,m}^V)u,v)
&=a_1^{(t)}(u,v)+a_1^{(t)}(\mathcal{C}_{t,m}u,v)\\
&=a_1^{(t)}(u,v)+c^{(t,m)}(u,v)\\
&=a(u,v)=\varphi(v)=\langle v,A_{t,m}u\rangle
\end{align*}
for each $v\in V$, i.e., $((I_V+\mathcal{C}_{t,m}^V)u,A_{t,m}u)\in A_1^{(t)}$. We have thus shown that
$u\in V$ and $(I_V+\mathcal{C}_{t,m}^V)u\in\DT(A_1^{(t)})$ 
with $A_1^{(t)}(I_V+\mathcal{C}_{t,m}^V)u=A_{t,m}u$ for all $u\in\DT(A_{t,m})$, i.e., $A_{t,m}\subseteq A_1^{(t)}(I_V+\mathcal{C}_{t,m}^V)$.
\\[0,2cm]
So, it only remains to check that
$\DT(A_1^{(t)}(I_V+\mathcal{C}_{t,m}^V))\subseteq\DT(A_{t,m})$. 
For that purpose, pick $u\in \DT(A_1^{(t)}(I_V+\mathcal{C}_{t,m}^V))$, i.e., $u\in V$
with $x:=(I_V+\mathcal{C}_{t,m}^V)u\in\DT(A_1^{(t)})$, and put $\varphi:=A_1^{(t)}(I_V+\mathcal{C}_{t,m}^V)u=A_1^{(t)}x\in H^*$.
Using the same computation as above, we then arrive at
\begin{align*}
\varphi(v)=\langle v,A_1^{(t)}x\rangle=a_1^{(t)}(x,v)=a_1^{(t)}((I_V+\mathcal{C}_{t,m}^V)u,v)=a(u,v)=a_{t,m}(u,v)
\end{align*}
for all $v\in V$ and we conclude that $u\in\DT(A_{t,m})$.
\\[0,2cm]
\emph{on f)}: We already know that (i) and (ii) are equivalent .
Furthermore, the addendum follows from part a) and the closed graph theorem.
Thanks to part d), $A_{t,m}$ is injective if and only if $I_V+\mathcal{C}_{t,m}^V$ is injective.
Moreover,
\begin{align*}
\ran(\mathcal{T}_t(I_V+\mathcal{C}_{t,m}^V))\cap\ran(j^\star)=j^\star(\ran(A_{t,m}))\subseteq \ran(j^\star).
\end{align*}
As $j^\star$ is injective, we have $\ran(A_{t,m})=H^\star$ if and only if $j^\star(\ran(A_{t,m}))=\ran(j^\star)$.
As a consequence, $\ran(A_{t,m})=H^\star$ if and only if $\ran(j^\star)\subseteq \ran(\mathcal{T}_t(I_V+\mathcal{C}_{t,m}^V))$.
This shows that (ii) and (iii) are equivalent.
\\[0,2cm]
\emph{on g)}:
We assume that problem \eqref{general 2} is $H$-well-posed.
Clearly, $\mathcal{J}$ is well-defined and injective due to part e) and f) above.
In addition, it is also surjective.
Indeed, pick $x\in\DT(A_1^{(t)})$.
Since $A_{t,m}$ is surjective, we may take $u\in\DT(A_{t,m})\subseteq V$ such that $A_{t,m}u=A_1^{(t)}x$.
We then obtain, employing part e),
\begin{align*}
A_1^{(t)}x=A_{t,m}u=A_1^{(t)}(I_V+\mathcal{C}_{t,m}^V)u,
\end{align*}
which implies $x=(I_V+\mathcal{C}_{t,m}^V)u$ due to the injectivity of $A_1^{(t)}$ (see part c)).
Thus $u$ belongs to
$\DT(A_1^{(t)}(I_V+\mathcal{C}_{t,m}^V))=\DT(A_{t,m})$ and satisfies $\mathcal{J}u=x$.

We estimate
\begin{align*}
\|\mathcal{J}u\|_{A_1^{(t)}}&=\|\mathcal{J}u\|_H+\|A_1^{(t)}\mathcal{J}u\|_{H^*}=\|\mathcal{J}u\|_H+\|A_1^{(t)}(I_V+\mathcal{C}_{t,m}^V)u\|_{H^*}\\
&\leq\|I_H+\mathcal{C}_{t,m}\|_{\Li(H)}\cdot\|u\|_H+\|A_{t,m}u\|_{H^*}\\
&\leq\upxi(\|u\|_H+\|A_{t,m}u\|_{H^*})=\upxi\|u\|_{A_{t,m}}
\end{align*}
for all $u\in\DT(A)$, where $\upxi:=\max\{1,\|I_H+\mathcal{C}_{t,m}\|_{\Li(H)}\}$.
Thanks to the open mapping theorem and the fact that $(A_1^{(t)},\|\cdot\|_{A_1^{(t)}})$ and $(A_{t,m},\|\cdot\|_{A_{t,m}})$ are Banach spaces by part a) and c) above, $\mathcal{J}$ is an isomorphism. 
\\[0,2cm]
\emph{on h)}:
 We first establish the claim for $j^\star A_{t,m}$ and assume that problem \eqref{general 2} is strongly well-posed.
By c) and e), $j^\star A_{t,m}\subseteq \mathcal{T}_t(I_V+\mathcal{C}_{t,m}^V)$. Since $\mathcal{T}_t(I_V+\mathcal{C}_{t,m}^V)$ is closed as a bounded operator from $V$ to $V^*$, we derive that $j^\star A_{t,m}$ is closable with $\overline{j^\star A_{t,m}}\subseteq \mathcal{T}_t(I_V+\mathcal{C}_{t,m}^V)$.
Fix $(u,\varphi)\in \mathcal{T}_t(I_V+\mathcal{C}_{t,m}^V)$ and pick a sequence $(\psi_n)_n$ from $H^*$ such that $\lim_{n\to\infty}j^\star(\psi_n)=\varphi$ in $V^*$.
This is in fact possible since  $j^\star$ has dense range.
As problem \eqref{general 2} is strongly well-posed, the operator $\mathcal{T}_t(I_V+\mathcal{C}_{t,m}^V)$ has a bounded inverse
thanks to Proposition \ref{well-posedness}.
We thus obtain
\begin{align*}
V\ni\,u_n:=(\mathcal{T}_t(I_V+\mathcal{C}_{t,m}^V))^{-1}(j^\star(\psi_n))\xrightarrow[n\to\infty]{}(\mathcal{T}_t(I_V+\mathcal{C}_{t,m}^V))^{-1}\varphi=u
\end{align*}
in $V$. By part a) of Proposition \ref{well-posedness}, $a_{t,m}(u_n,v)=j^\star(\psi_n)(v)=\psi_n(j(v))=\psi_n(v)$ for all $v\in V$ and we therefore
have $u_n\in\DT(A_{t,m})$
with $A_{t,m}u_n=\psi_n$ for every $n\in\N$. Thus, we finally deduce
\begin{align*}
j^\star A_{t,m}\ni\,(u_n,j^\star(\psi_n))\xrightarrow[n\to\infty]{}(u,\varphi)
\end{align*}
in $V\times V^*$. This shows $\overline{j^\star A_{t,m}}\supseteq \mathcal{T}_t(I_V+\mathcal{C}_{t,m}^V)$.

The proof for $j^\star A_{1}^{(t)}$ is similar, but simpler.
\end{proof}

\vspace*{1ex}

Assume for a moment that problem \eqref{general 2} is strongly well-posed for \emph{all} $(t,m)\in E\times U$. For fixed $\varphi\in H^\star$,
we then may consider the by now well-defined parameter-to-state operator
\begin{align}\label{parameter-to-state operator}
E\times U\rightarrow\DT(A_{t,m});\, (t,m)\mapsto u_{t,m,\varphi}:=A_{t,m}^{-1}\varphi.
\end{align}
The inverse problem w.r.t $m$ arising from problem \eqref{general 2} consists in reconstructing $m$ from $u_{t,m,\varphi}$ for
fixed $\varphi\in H^\star$ and $t\in E$. Thanks to Theorem \ref{general factorisation theorem}
we obtain the following commutative diagram.

\begin{centering}
\begin{align}\notag
\begin{xy}
  \xymatrix{
     H\supseteq \DT(A_{t,m})\ar[d]_{I_V+\mathcal{C}_{t,m}^V}^{\cong} \ar[r]^{A} & H^*   \\
     H\supseteq \DT(A_1^{(t)})\ar[ru]_{A_1^{(t)}} &
  }
 \end{xy}
\end{align}
\end{centering}

Put another way, the operator $A=A_{t,m}$ factorises into an operator that does not depend at all on the parameter $m$ and the isomorphism
$I_V+\mathcal{C}_{t,m}^V$ on the solution space
$V$ that encompasses the dependence on $m$. This explains why the properties of the operator $I_V+\mathcal{C}_{t,m}^V$
are crucial for the analytic features of the parameter-to-state operator as explored in section \ref{inverse problem} below.

As a direct consequence of Theorem \ref{general factorisation theorem} we finally see that in an important situation the terms
of strong well-posedness and $H$-well-posedness coincide.

\vspace*{1ex}

\begin{corollary}
Besides the premises of  Theorem \ref{general factorisation theorem}, suppose that $j$ is compact. Then problem \eqref{general 2}
is strongly well-posed if and only if it is $H$-well-posed.
\end{corollary}

\vspace*{1ex}

\begin{proof}
If problem \eqref{general 2} is $H$-well-posed, then $I_V+\mathcal{C}_{t,m}^{V}$ is injective.
Hence, problem \eqref{general 2} is strongly well-posed 
since we can apply part c) of Proposition \ref{well-posedness},
thanks to the compactness of $j$.
\end{proof}

\section{Well-posedness results for the variational problem \eqref{general 2}} \label{section_ex_uni}

We are now able to formulate and to prove our two main well-posedness results for the variational problem \eqref{general 2}. We have a local and a global well-posedness result for problem \eqref{general 2}
in the sense that in the global version we can establish, under appropriate conditions, well-posedness of problem \eqref{general 2} w.r.t. the entire
parameter range $E\times U$ (see Remark \ref{Bedeutung von tilde U} below), whereas in the local version we may guarantee well-posedness only on a suitable open subset of $E\times U$. We start with the \emph{global version}.

\vspace*{1ex}

\begin{theorem}\label{general global version}
Let
$\mathcal C_{t,m}:H\rightarrow H$ be the operator considered in Lemma \ref{general Ausgangslemma} and
 $\mathcal C_{t,m}^V$ its part in $V$.
Assume that the inclusion $V\subseteq H$ is compact and that for all $(t,m)\in E\times \widetilde U$ and all $u\in V$ the implication 
 \begin{align}\label{general injectivity assumption}
\bigg(\forall\,w\in W:\,a_1^{(t)}(u,w)+c^{(t,m)}(u,w)=0\bigg) \ \Longrightarrow \ u=0
 \end{align}
 is valid, where $\widetilde U$ is a non-empty subset of $U$. 
Then the following claims hold.
\begin{enumerate}
\item There exists a set $\mathcal U\subseteq E\times U$ open in $E\times X$ and
 containing $E\times\widetilde U$ such that $I_V+\mathcal C_{t,m}^V$ is invertible for all
$(t,m)\in \mathcal U$ and the inverse depends continuously on $(t,m)\in\mathcal U$.
Furthermore, for each $t\in E$ there exists a set $\mathcal{U}_t\subseteq U$ open in $X$ containing $\widetilde U$ such that
$I_V+\mathcal C_{t,m}^V$ is invertible for all $m\in \mathcal U_t$.
\item For each antilinear functional $\varphi\in W^*$ and each pair $(t,m)\in\mathcal U$ there exists a unique $u\in V$ such that
  \begin{align*}
 \forall\,w\in W:\,a_1^{(t)}(u,w)+c^{(t,m)}(u,w)=\varphi(w)
 \end{align*}
 and this unique $u$ depends continuously on $t$, $m$, and $\varphi$.
In addition, we have
 \begin{align}\label{general 3}
 \|u\|_V\leq\frac{1}{c(t)}\|(I_V+\mathcal C_{t,m}^V)^{-1}\|_{\Li(V)}\|\varphi\|_{W^*}.
 \end{align}
The analogous conclusions are valid for fixed $t\in E$ and $m\in\mathcal{U}_t$.
\end{enumerate}
\end{theorem}

\begin{proof}
Let $(t,m)\in E\times \widetilde U$ be arbitrary and $u\in V$.
By \eqref{general injectivity assumption} and by part b) and c) of Proposition \ref{well-posedness},
we obtain that
$I_V+\mathcal C_{t,m}^V$ is an isomorphism.
 As
 $\mathcal C_{t,m}^V$ depends continuously on $(t,m)\in E\times U$,
 the function
 \begin{align*}
 F:E\times U\rightarrow\Li(V); \, (t,m)\mapsto I_V+\mathcal C_{t,m}^V
 \end{align*}
 is continuous. Summarizing, $\mathcal U:=F^{-1}(\Is(V))$ is a subset of $E\times U$ open w.r.t. the relative topology on $E\times U$
 containing $E\times\widetilde U$.
 But as $U$ is open in $X$, we deduce that $\mathcal U$ is open in $E\times X$.
 Clearly, $I_V+\mathcal C_{t,m}^V$ as well as its inverse depend continuously on $(t,m)\in\mathcal{U}$.
\\[0,2cm]
Let $\varphi\in W^*$ and $(t,m)\in\mathcal U$ be arbitrary.
 Thanks to Proposition \ref{well-posedness}, problem \eqref{general 2} has now precisely one solution $u\in V$
 given by $u=\big(I_V+\mathcal C_{t,m}^V\big)^{-1}\mathcal T_t^{-1}(\varphi)\in V$.
 Consequently, such a solution necessarily satisfies 
 \begin{align*}
\|u\|_V
&=\big\lVert\big(I_V+\mathcal C_{t,m}^V\big)^{-1}\mathcal T_t^{-1}(\varphi)\big\rVert_V\\
&\leq\big\lVert\big(I_V+\mathcal C_{t,m}^V\big)^{-1}\big\rVert_{\Li(V)}\cdot\|\mathcal T_t^{-1}\|_{\Li(W^*,V)}\|\varphi\|_{W^*}\\
&\leq\frac{1}{c(t)}\big\lVert\big(I_V+\mathcal C_{t,m}^V\big)^{-1}\big\rVert_{\Li(V)}\|\varphi\|_{W^*},
\end{align*}
which shows inequality \eqref{general 3}.
In addition, it is easy to show that $u$ depends continuously on $t$, $m$ and $\varphi$
by using this representation for $u$ (cf. the arguments used to establish Lemma \ref{general Ausgangslemma}).
\\[0,2cm]
Finally, for fixed $t\in E$ we may apply the results shown so far for $E_t:=\{t\}$ instead of $E$ in order to establish the remaining assertions.
\end{proof}

\vspace*{1ex}

\begin{remark}\label{Bedeutung von tilde U}
Observe that in Theorem \ref{global version} the choice $\widetilde U=U$ is possible.
Therefore we obtain global well-posedness, i.e., for all parameter values $(t,m)\in E\times U$ provided that condition \eqref{general injectivity assumption}
is satisfied for $\widetilde U=U$. The conceptual advantage that justifies the introduction of the set $\widetilde U$ in the formulation of Theorem \ref{global version} consists in the fact that it suffices especially to check condition \eqref{general injectivity assumption}
on the set $\widetilde U$ for a fixed $t\in E$, where $\widetilde U$ needs not to be open,
to gain for free well-posedness on a larger set $\mathcal U_t$ open in $X$.
Hence, $\mathcal U_t$ is best suited for differential calculus.
This conclusion plays a vital role in the treatment of some examples, including the inverse problem of THz tomography, in a forthcoming paper.
\end{remark}

\vspace*{1ex}

We now come to the \emph{local version}.

\vspace*{1ex}

\begin{theorem}\label{general local version}
Let
$\mathcal C_{t,m}:H\rightarrow H$ and $\mathcal C_{t,m}^V:V\rightarrow V$ be as before.
Assume that there exists a net $(t_\alpha)_{\alpha\in\mathbb{A}}$  ($\mathbb{A}$ a directed set) in $E$ with
\begin{align}\label{general A2}
\lim_{\alpha\in\mathbb{A}}\frac{M(t_\alpha,m)}{c(t_\alpha)}=0
\end{align}
for all $m\in U$.
Then there exists a non-empty set $\mathcal O\subseteq E\times U$ open in $E\times X$ with the following properties.
\begin{enumerate}
\item For all $m\in U$ there exists a non-empty, open subset $\mathcal O_m\subseteq E$ such that $\mathcal O_m\times\{m\}\subseteq\mathcal O$.
\item The operator $I_H+\mathcal C_{t,m}$ is invertible for all $(t,m)\in\mathcal O$.
\item The operator $I_V+\mathcal C_{t,m}^V$ is invertible as an element of $\Li(V)$ for all $(t,m)\in\mathcal O$
 and both this operator and its inverse depend continuously on $(t,m)\in\mathcal O$.
\item For all $(t,m)\in\mathcal O$ and each antilinear functional $\varphi\in W^*$ there exists a unique $u\in V$ such that
  \begin{align*}
 \forall\,w\in W:\,a_1^{(t)}(u,w)+c_2^{(t,m)}(u,w)=\varphi(w)
 \end{align*}
 and this unique $u$ depends continuously on $t$, $m$ and $\varphi$.
 In addition, we have
 \begin{align}\notag
 \|u\|_V\leq\frac{1}{c(t)}\|(I_V+\mathcal C_{t,m}^V)^{-1}\|_{\Li(V)}\|\varphi\|_{W^*}.
 \end{align}
\end{enumerate} 
\end{theorem}

\begin{proof}
Using $\mathcal{C}_{t,m}(H)\subseteq V$ and part  c) of Lemma \ref{general Ausgangslemma}, we derive
\begin{align*}
 \|\mathcal C_{t,m}x\|_H\leq\gamma \|\mathcal C_{t,m}x\|_V\leq\frac{\gamma M(t,m)}{c(t)}\|x\|_H,
\end{align*}
 which implies
 \begin{align*}
 \|\mathcal C_{t,m}\|_{\Li(H)}\leq\frac{\gamma M(t,m)}{c(t)}.
 \end{align*}
Employing hypothesis \eqref{general A2}, we derive
\begin{align*}
\|\mathcal C_{t_\alpha,m}\|_{\Li(H)}\leq \frac{\gamma M(t_\alpha,m)}{c(t_\alpha)}
\xrightarrow[\alpha\in\mathbb A]{}0
\end{align*}
for all $m\in U$,
which yields
\begin{align*}
I_H+\mathcal C_{t_\alpha,m}\xrightarrow[\alpha\in\mathbb{A}]{}I_H\in\Is(H).
\end{align*}
Since $\Is(H)$ is an open subset of $\Li(H)$ and $\mathcal C:E\times U\rightarrow\Li(H);\, (t,m)\mapsto\mathcal{C}_{t,m}$ is continuous,
we deduce that for fixed $m\in U$ the set
\begin{align*}
\mathcal O_m:=\{t\in E:\,I_H+\mathcal C_{t,m}\in\Is(H)\}
\end{align*}
is a non-empty open subset of $E$ as well as that the set
\begin{align*}
\mathcal O:=\{(t,m)\in E\times U:\,I_H+\mathcal C_{t,m}\in\Is(H)\}
\end{align*}
is non-empty and open w.r.t. the relative topology on $E\times U$.
As $U$ is open in $X$, we infer that $\mathcal{O}$ is an open subset of $E\times X$.
Clearly, $\mathcal O_m\times\{m\}\subseteq\mathcal O$ for all $m\in U$.
This shows part a) and b).

The remaining assertions can now be deduced essentially as in the proof of Theorem \ref{general global version} as soon as we will have shown that
$I_V+\mathcal C_{t,m}^V$ is invertible for all $(t,m)\in\mathcal O$.
By part b), the operator $I_V+\mathcal C_{t,m}^V$ is injective for $(t,m)\in\mathcal O$. Let $\widetilde v\in V$ be arbitrary.
Thanks to part b), there exists a $v\in H$ such that $(I_H+\mathcal C_{t,m})(v)=\widetilde v$.
This last equality is equivalent to $v=\widetilde v-\mathcal C_{t,m}v$ and we infer that $v\in V$ because of $\mathcal C_{t,m}(H)\subseteq V$.
As a result, $I_V+\lambda(t)\mathcal C_{t,m}^V$ is also surjective, thus continuously invertible by the open mapping theorem. 
\end{proof}

\vspace*{1ex}

\begin{remark}
As we said before, one may interpret problem \eqref{general 2} resp. \eqref{2} as the weak formulation of an elliptic boundary value problem, where $W$ serves as a space of test functions. 
Roughly speaking, the lower order terms of the corresponding differential operator are encoded in the form $c^{(t,m)}$ and they depend on the parameter $m$,
while $a_1^{(t)}$ essentially describes the highest order terms.
Consequently, assumption \eqref{general A2} in the local well-posedness result is a kind of smallness condition w.r.t. the highest order terms imposed on the lower order terms of the involved differential operator.
\end{remark}

\vspace*{1ex}

At the end of this section, we want to briefly discuss the case we are particularly interested in, namely
\begin{align*}
\mathfrak{c}(t,m)=\lambda(t)\mathfrak{a}_2(m)
\end{align*}
for all $t\in E$ and $m\in U$. In that case we obtain, following the same line of argument as above, the subsequent slightly more precise versions of the previous results.

\vspace*{1ex}

\begin{lemma}\label{Ausgangslemma}
For each pair $(t,m)\in E\times U$ there exists a unique bounded operator
$\mathcal A_{t,m}:H\rightarrow H$ with $\mathcal A_{t,m}(H)\subseteq V$ and with
 \begin{align}\label{1}
 a_1^{(t)}(\mathcal A_{t,m}x,w)=a_2^{(m)}(x,w)
 \end{align}
 for every $x\in H$ and each $w\in W$.
 In addition, the following assertions are valid.
\begin{enumerate}
\item The mapping $\mathcal A:E\times U\rightarrow\Li(H), \ (t,m)\mapsto\mathcal A_{t,m}$ is continuous.
\item The part of $\mathcal A_{t,m}$ in $V$, i.e., the linear operator
 \begin{align*}
 \mathcal A_{t,m}^V:V\rightarrow V, \ v\mapsto\mathcal A_{t,m}v
 \end{align*}
 is bounded and the mapping $\mathcal A^V:E\times U\rightarrow\Li(V), \ (t,m)\mapsto\mathcal C_{t,m}^V$ is continuous.
\item We have $\|\mathcal A_{t,m}x\|_V\leq\frac{M(m)}{c(t)}\cdot\|x\|_H$ for each $x\in H$.
\item The operators $\mathcal A_{t,m}$ and $\mathcal{A}_{t,m}^V$ are both compact if the embedding $j:V\rightarrow H$ is compact.
\end{enumerate}
\end{lemma}

\vspace*{1ex}

\begin{theorem}\label{global version}
Let
$\mathcal A_{t,m}:H\rightarrow H$ be the operator considered in Lemma \ref{Ausgangslemma}.
We further consider the part of it in $V$.
Assume that the inclusion $V\subseteq H$ is compact and that for all $(t,m)\in E\times \widetilde U$ and all $u\in V$ the implication 
 \begin{align}\label{injectivity assumption}
\bigg(\forall\,w\in W:\,a_1^{(t)}(u,w)+\lambda(t) a_2^{(m)}(u,w)=0\bigg) \ \Longrightarrow \ u=0
 \end{align}
 is valid, where $\widetilde U$ is a non-empty subset of $U$. 
Then the following claims hold.
\begin{enumerate}
\item There exists a set $\mathcal U\subseteq E\times U$ open in $E\times X$ and
containing $E\times\widetilde U$ such that $I_V+\lambda(t)\mathcal A_{t,m}^V$ is invertible for all $(t,m)\in \mathcal U$ and its inverse depends continuously on $(t,m)\in\mathcal U$.
Furthermore, for each $t\in E$ there exists a set $\mathcal{U}_t\subseteq U$ open in $X$ and 
containing $\widetilde U$ such that
$I_V+\lambda(t)\mathcal A_{t,m}^V$ is invertible for all $m\in \mathcal U_t$.
\item For each antilinear functional $\varphi\in W^*$ and each pair $(t,m)\in\mathcal U$ there exists a unique $u\in V$ such that
  \begin{align*}
 \forall\,w\in W:\,a_1^{(t)}(u,w)+\lambda(t) a_2^{(m)}(u,w)=\varphi(w)
 \end{align*}
 and this unique $u$ depends continuously on $t$, $m$, and $\varphi$.
In addition, we have
 \begin{align}\label{3}
 \|u\|_V\leq\frac{1}{c(t)}\|(I_V+\lambda(t)\mathcal A_{t,m}^V)^{-1}\|_{\Li(V)}\|\varphi\|_{W^*}.
 \end{align}
The analogous conclusions are valid for fixed $t\in E$ and $m\in\mathcal{U}_t$.
\end{enumerate}
\end{theorem}

\vspace*{1ex}

\begin{theorem}\label{local version}
Let
$\mathcal A_{t,m}:H\rightarrow H$ and $\mathcal A_{t,m}^V:V\rightarrow V$ be as before.
Assume that there exists a net $(t_\alpha)_{\alpha\in\mathbb{A}}$  ($\mathbb{A}$ a directed set) in $E$ with
\begin{align}\label{A2}
\lim_{\alpha\in\mathbb{A}}\frac{\lambda(t_\alpha)}{c(t_\alpha)}=0.
\end{align}
Then there exists a non-empty set $\mathcal O\subseteq E\times U$ open in $E\times X$ with the following properties.
\begin{enumerate}
\item For all $m\in U$ there exists a non-empty, open set $\mathcal O_m\subseteq E$ such that $\mathcal O_m\times\{m\}\subseteq\mathcal O$.
\item The operator $I_H+\lambda(t)\mathcal A_{t,m}$ is invertible for all $(t,m)\in\mathcal O$.
\item The operator $I_V+\lambda(t)\mathcal A_{t,m}^V$ is invertible as an element of $\Li(V)$ for all $(t,m)\in\mathcal O$
 and both this operator and its inverse depend continuously on $(t,m)\in\mathcal O$.
\item For all $(t,m)\in\mathcal O$ and each antilinear functional $\varphi\in W^*$ there exists a unique $u\in V$ such that
  \begin{align*}
 \forall\,w\in W:\,a_1^{(t)}(u,w)+\lambda(t) a_2^{(m)}(u,w)=\varphi(w)
 \end{align*}
 and this unique $u$ depends continuously on $t$, $m$ and $\varphi$.
 In addition, we have
 \begin{align}\notag
 \|u\|_V\leq\frac{1}{c(t)}\|(I_V+\lambda(t)\mathcal A_{t,m}^V)^{-1}\|_{\Li(V)}\|\varphi\|_{W^*}.
 \end{align}
\end{enumerate} 
\end{theorem}

\section{Inverse problems} \label{inverse problem}

Assuming the well-posedness of problem \eqref{general 2}, we will now explore the analytic properties of various parameter-to-state operators.

\subsection{Inverse problem with respect to the parameter $m$}
 We first consider the inverse problem with respect to the parameter $m$.

\vspace*{1ex}  
 
 \begin{theorem}\label{regularity}
 Let $\nu\in\N\cup\{\infty\}$, $t\in E$ be fixed, and $\mathcal{G}_t$ a non-empty, open subset of $U$.
 Assume that problem \eqref{general 2} is strongly well-posed for all $m\in \mathcal{G}_t$. We further consider a mapping
 \begin{align*}
 \Phi:\mathcal{G}_t\rightarrow W^*;\,m\mapsto\varphi_m
 \end{align*}
 as well as the parameter-to-state operator
 \begin{align*}
 S:\mathcal{G}_t\rightarrow V;\, m\mapsto u_{m},
 \end{align*}
 where $u_m=u_{t,m,\varphi_m}$ is the unique solution $u\in V$ of the problem
 \begin{align*}
  \forall\,w\in W:\,a_1^{(t)}(u,w)+c^{(t,m)}(u,w)=\varphi_m(w)=\langle w,\Phi(m)\rangle.
 \end{align*}
 \begin{enumerate}
  \item If $\Phi$ and $\mathfrak{c}_t:=\mathfrak{c}(t,\cdot)$ are both $\nu$-times (continuously) Fr\'echet-differentiable on $\mathcal{G}_t$, then $S$ is also $\nu$-times (continuously)
  Fr\'echet-differentiable on $\mathcal{G}_t$.
  \item If $\Phi$ and $\mathfrak{c}(t,\cdot)$ are both analytic on $\mathcal{G}_t$ (in the sense that they are locally given by their respective Taylor series expansion, see \cite{Whittlesey}), then $S$ is also analytic on $\mathcal{G}_t$.
 \end{enumerate}
 \end{theorem}

\vspace*{1ex} 

\begin{proof}
Using part a) of Proposition \ref{well-posedness} and the construction of $\mathcal{C}_{t,m}^V$, we see that
\begin{align}\label{representation solution-to-parameter operator 2}
S(m)=(I_V+\mathcal{C}_{t,m}^V)^{-1}\mathcal{T}_t^{-1}\Phi(m)
=(I_V+\mathcal{T}_t^{-1}\mathcal{B}_{t,m}j)^{-1}\mathcal{T}_t^{-1}\Phi(m)
\end{align}
for all $m\in\mathcal{G}_t$, where
\begin{align*}
\mathcal B_{t,m}:H\rightarrow W^*, \ x\mapsto c^{(t,m)}(x,\,\cdot\,).
\end{align*}
It is well-known and easy to check that the operator
\begin{align*}
\Xi:\mathcal{S}(H\times W,\mathbb K)\rightarrow\Li(H,W^*);\,\mathfrak{d}\mapsto\Xi(\mathfrak{d}),
\end{align*}
where
$\Xi(\mathfrak{d})\lbrack x\rbrack=\mathfrak{d}(x,\cdot)$
for $x\in H$, is a well-defined isometric isomorphism.
We further consider the following bounded, linear operators
\begin{align*}
L_{\mathcal{T}_t^{-1}}:\Li(H,W^*)\rightarrow\Li(H,V);\,T\mapsto\mathcal{T}_t^{-1}T,
\end{align*}
\begin{align*}
R_{\mathcal{T}_t^{-1}}:\Li(V)\rightarrow\Li(W^*,V);\,T\mapsto T\mathcal{T}_t^{-1},
\end{align*}
and
\begin{align*}
R_j:\Li(H,V)\rightarrow\Li(V);\, T\mapsto Tj
\end{align*}
as well as the continuous function
\begin{align*}
\mathfrak{c}_t=\mathfrak{c}(t,\cdot):\mathcal{G}_t\rightarrow\mathcal{S}(H\times W,\mathbb{K});\,m\mapsto c^{(t,m)}=\mathfrak{c}(t,m),
\end{align*}
the translation
\begin{align*}
\tau:\Li(V)\rightarrow\Li(V);\, T\mapsto I_V+T,
\end{align*}
the bounded, bilinear mapping
\begin{align*}
\mathfrak{b}:\Li(W^*,V)\times W^*\rightarrow V;\,(T,\varphi)\mapsto T\varphi,
\end{align*}
and the inversion
\begin{align*}
\inv_{V}:\Is(V)\rightarrow\Is(V);\, T\mapsto T^{-1}.
\end{align*}
We put
\begin{align*}
\widetilde S:=R_{\mathcal{T}_t^{-1}}\circ\inv_V\circ\tau\circ R_j\circ L_{\mathcal{T}_t^{-1}}\circ\Xi\circ\mathfrak{c}_t:\mathcal{G}_t\rightarrow\Li(W^*,V)
\end{align*}
and claim that
\begin{align}\label{representation parameter-to-state operator}
S(m)=\mathfrak{b}(\widetilde S(m),\Phi(m))
\end{align}
for all $m\in\mathcal{G}_t$.
Since bounded (multi)linear operators, translations as well as the inversion of isomorphisms (see \cite[p. 1080]{Whittlesey}) are analytic functions,
the chain rule (see \cite[p. 1079]{Whittlesey} and \cite[Theorem VII.5.7]{AmannEscherII}) gives us the assertions as soon as we will have shown \eqref{representation parameter-to-state operator}.
Take $m\in\mathcal{G}_t$. By definition,
\begin{align}\label{****}
(\Xi\circ\mathfrak{c}_t)(m)x=c^{(t,m)}(x,\cdot)=\mathcal{B}_{t,m}(x)
\end{align}
for all $x\in H$, i.e., $(\Xi\circ\mathfrak{c}_t)(m)=\mathcal{B}_{t,m}=\mathcal{B}_t(m)$, where
\begin{align*}
\mathcal{B}_t:\mathcal{G}_t\rightarrow\Li(H,W^*);\,m\mapsto\mathcal{B}_{t,m}.
\end{align*}
As a result, we infer
\begin{align}\label{*}
\widetilde S(m)
=R_{\mathcal{T}_t^{-1}}(\inv_V(\tau(R_j(L_{\mathcal{T}_t^{-1}}(\mathcal{B}_{t,m})))))=(I_V+\mathcal{T}_t^{-1}\mathcal{B}_{t,m}j)^{-1}\mathcal{T}_t^{-1},
\end{align}
which finally yields
\begin{align*}
\mathfrak{b}(\widetilde S(m),\Phi(m))=(I_V+\mathcal{T}_t^{-1}\mathcal{B}_{t,m}j)^{-1}\mathcal{T}_t^{-1}\Phi(m)=S(m)
\end{align*}
due to \eqref{representation solution-to-parameter operator 2}.
\end{proof}

\vspace*{1ex}

\begin{remark}
One might ask whether or not it is necessary to assume that $\mathfrak{c}_t$ and $\Phi$ are Fr\'echet-differentiable in order to make sure that the considered parameter-to-state operator is Fr\'echet-differentiable. In general this is not the case. 
Indeed, if, for instance, $\Phi$ is identical zero, $S=0$ will trivially be Fr\'echet-differentiable, independently of the differentiability properties of
$\mathfrak{c}_t$.
However, formula \eqref{*} reveals that $\widetilde S$ is $\nu$-times (continuously) Fr\'echet-differentiable resp. analytic if and only if this holds for the mapping
\begin{align*}
\mathcal{G}_t\rightarrow\Li(V,W^*);\, m\mapsto\mathcal{B}_t(m)j.
\end{align*}
Moreover, note that if both $\mathfrak{c}_t$ and $S$ are Fr\'echet-differentiable, it easily follows from
\eqref{representation solution-to-parameter operator 2} and \eqref{****}
that $\Phi$ must be Fr\'echet-differentiable, too.
\end{remark}

\vspace*{1ex}

Assume that the hypotheses from Theorem \ref{regularity} hold.
Using the representation \eqref{representation parameter-to-state operator} we may compute the Fr\'echet-derivative of the parameter-to-state operator $S$. For that purpose, recall (see, e.g., \cite[Satz VII.7.2]{AmannEscherII}) that
\begin{align*}
\mathrm D_{\mathcal F}\inv_V:\Is(V)\rightarrow\Li(\Li(V));\, T\mapsto -L_{T^{-1}}\circ R_{T^{-1}},
\end{align*}
where
\begin{align*}
R_{T^{-1}}:\Li(V)\rightarrow\Li(V);\, \Psi\mapsto\Psi T^{-1}
\end{align*}
and
\begin{align*}
L_{T^{-1}}:\Li(V)\rightarrow\Li(V);\, \Psi\mapsto T^{-1}\Psi. 
\end{align*}
Thus, we obtain, using the chain rule,
\begin{align*}
\mathrm D_{\mathcal F}\widetilde S(m)
&=R_{\mathcal{T}_t^{-1}}(\mathrm D_{\mathcal F}\inv_V)\Big((\tau\circ R_j\circ L_{\mathcal{T}_t^{-1}}\circ\Xi\circ\mathfrak{c}_t)(m))\Big)
R_jL_{\mathcal{T}_t^{-1}}(\mathrm D_{\mathcal F}\mathcal B_t)(m)
\end{align*}
and hence
\begin{align*}
\mathrm D_{\mathcal F}\widetilde S(m)\lbrack\widetilde m\rbrack
&=R_{\mathcal{T}_t^{-1}}(\mathrm D_{\mathcal F}\inv_V)\Big\lbrack I_V+\mathcal{T}_t^{-1}\mathcal{B}_{t,m}j\Big\rbrack
\Big(\mathcal{T}_t^{-1}((\mathrm D_{\mathcal F}\mathcal B_t)(m)\lbrack\widetilde m\rbrack)j\Big)\\
&=-\Big(I_V+\mathcal{T}_t^{-1}\mathcal{B}_{t,m}j\Big)^{-1}
\mathcal{T}_t^{-1}\big((\mathrm D_{\mathcal F}\mathcal B_t)(m)\lbrack\widetilde m\rbrack\big)j
\Big(I_V+\mathcal{T}_t^{-1}\mathcal{B}_{t,m}j\Big)^{-1}
\mathcal{T}_t^{-1}
\end{align*}
for all $m\in\mathcal{G}_t$ and all $\widetilde m\in X$.

In order to proceed, we need the subsequent product rule for the Fr\'echet-derivative, which can be easily derived:
\\[0,2cm]  
Let $X_0$, $X_1$, $X_2$, and $X_3$ be Banach spaces, $\Omega\subseteq X_0$ open and non-empty, $x_0\in\Omega$,
$f:\Omega\rightarrow X_1$ and $g:\Omega\rightarrow X_2$ functions, which are Fr\'echet-differentiable at $x_0$, 
and $\beta:X_1\times X_2\rightarrow X_3$ a bounded bilinear mapping.
Then the function
 \begin{align*}
 \beta(f,h):\Omega\rightarrow X_3;\, x\mapsto \beta(f(x),g(x))
 \end{align*}
 is Fr\'echet-differentiable at $x_0$ with
 \begin{align*}
 \mathrm D_{\mathcal{F}}(\beta(f,h))(x_0)\xi
 =\beta((\mathrm D_{\mathcal{F}}f)(x_0)\xi,g(x_0))+\big(\beta(f(x_0),(\mathrm D_{\mathcal{F}}g)(x_0)\xi)\big)
 \end{align*}
for all $\xi\in X_0$.
\\[0,2cm]  
Employing this product rule and \eqref{representation solution-to-parameter operator 2}, we calculate
\begin{align*}
\mathrm D_{\mathcal F}S(m)\lbrack\widetilde m\rbrack
&=\mathfrak{b}(\mathrm D_{\mathcal F}\widetilde S(m)\lbrack\widetilde m\rbrack,\Phi(m))
+\mathfrak{b}(\widetilde S(m),\mathrm D_{\mathcal F}\Phi(m)\lbrack\widetilde m\rbrack)\\
&=-(I_V+\mathcal{T}_t^{-1}\mathcal{B}_{t,m}j)^{-1}
\mathcal{T}_t^{-1}\big((\mathrm D_{\mathcal F}\mathcal B_t)(m)\lbrack\widetilde m\rbrack\big)S(m)
+\widetilde S(m)\big(\mathrm D_{\mathcal F}\Phi(m)\lbrack\widetilde m\rbrack\big)
\end{align*}
for  $m\in\mathcal{G}_t$ and  $\widetilde m\in X$.
As problem \eqref{general 2} is, by assumption, strongly well-posed for $m\in\mathcal{G}_t$, we may restate this result, using part a) of Proposition \ref{well-posedness} as well as \eqref{*}, as follows: 
$\mathrm D_{\mathcal F}S(m)\lbrack\widetilde m\rbrack$ is the unique element $u\in V$ such that
\begin{align}\label{derivative parameter-to-state operator}
a_1^{(t)}(u,w)+c^{(t,m)}(u,w)
=\Big\langle w,\mathrm D_{\mathcal F}\Phi(m)\lbrack\widetilde m\rbrack-\big((\mathrm D_{\mathcal F}\mathcal B_t)(m)\lbrack\widetilde m\rbrack\big)S(m)\Big\rangle
\end{align}
for all $w\in W$. This result has the following remarkable consequence.

\vspace*{1ex}  
 
 \begin{theorem}\label{tcc for general S}
 Let $t\in E$ be fixed and $\mathcal{G}_t$ a non-empty, open subset of $U$.
 Assume that problem \eqref{general 2} is strongly well-posed for all $m\in \mathcal{G}_t$. We further consider a continuous affine linear mapping
 \begin{align*}
 \Phi:X\rightarrow W^*;\,m\mapsto\varphi_m
 \end{align*}
 as well as the parameter-to-state operator
 \begin{align*}
 S:\mathcal{G}_t\rightarrow V;\, m\mapsto u_{m},
 \end{align*}
 where $u_m=u_{t,m,\varphi_m}$ is the unique solution $u\in V$ of the problem
 \begin{align*}
  \forall\,w\in W:\,a_1^{(t)}(u,w)+c^{(t,m)}(u,w)=\varphi_m(w)=\langle w,\Phi(m)\rangle.
 \end{align*}
Moreover, we assume that $\mathfrak{c}_t$ is the restriction of a continuous affine linear mapping defined on $X$.
 In particular, $S$ is continuously Fr\'echet-differentiable on $\mathcal{G}_t$ thanks to Theorem \ref{regularity}.
Then, for each $m_0\in \mathcal G_t$ and every $\kappa\in(0,1)$ there exists a constant $\varrho=\varrho(m_0,\kappa)>0$ such that $B_\varrho(m_0)\subseteq \mathcal G_t$,
 the Fr\'echet-derivative $\mathrm D_{\mathcal F}S$ of $S$ is bounded on $B_\varrho(m_0)$ and $S$ satisfies on $B_\varrho(m_0)$ a
 $\kappa$-tangential cone condition w.r.t. both $\|\cdot\|_H$ and $\|\cdot\|_V$, i.e., we have
 \begin{align}\label{general tcc}
 \|S(m_1)-S(m_2)-(\mathrm D_{\mathcal F}S(m_2))\lbrack m_1-m_2\rbrack\|_H\leq\kappa\|S(m_1)-S(m_2)\|_H
 \end{align}
 and
  \begin{align}\label{general tcc in V}
 \|S(m_1)-S(m_2)-(\mathrm D_{\mathcal F}S(m_2))\lbrack m_1-m_2\rbrack\|_V\leq\kappa\|S(m_1)-S(m_2)\|_V
 \end{align}
 for all $m_1,m_2\in B_\varrho(m_0)$.
 \end{theorem}
 
 \vspace*{1ex}
 
 \begin{proof}
 Let $m\in \mathcal G_t$, $h\in X\setminus\{0\}$ such that $m+h\in \mathcal G_t$, let $w\in W$, and put
 $u:=S(m+h)-S(m)-(\mathrm D_{\mathcal F}S(m))\lbrack h\rbrack$.
 Using \eqref{derivative parameter-to-state operator} and \eqref{****}, we deduce
 
 \begin{align*}
 a_1^{(t)}(u,w)+c^{(t,m)}(u,w)
 =&a_1^{(t)}(S(m+h),w)+c^{(t,m+h)}(S(m+h),w)\\
 &-c^{(t,m+h)}(S(m+h),w)+c^{(t,m)}(S(m+h),w)\\
 &-\Big(a_1^{(t)}(S(m),w)+c^{(t,m)}(S(m),w)\Big)\\
 &-\bigg(a_1^{(t)}((\mathrm D_{\mathcal F}S(m))\lbrack h\rbrack,w)+c^{(t,m)}((\mathrm D_{\mathcal F}S(m))\lbrack h\rbrack,w))\bigg)\\
 =&\langle w,\Phi(m+h)\rangle-\langle w,\Phi(m)\rangle\\
 &-c^{(t,m+h)}(S(m+h),w)+c^{(t,m)}(S(m+h),w)\\
 &-\Big\langle w,\mathrm D_{\mathcal F}\Phi(m)\lbrack h\rbrack-\big((\mathrm D_{\mathcal F}\mathcal B_t)(m)\lbrack h\rbrack\big)S(m)\Big\rangle\\
 =&\Big\langle w,\Phi(m+h)-\Phi(m)-\mathrm D_{\mathcal F}\Phi(m)\lbrack h\rbrack\Big\rangle\\
 &-\bigg\langle w,\Big(\mathcal{B}_t(m+h)-\mathcal{B}_t(m)-(\mathrm D_{\mathcal F}\mathcal B_t)(m)\lbrack h\rbrack\Big)(S(m+h))\bigg\rangle\\
 &+\Big\langle w,\big((\mathrm D_{\mathcal F}\mathcal B_t)(m)\lbrack h\rbrack\big)S(m)
 -\big((\mathrm D_{\mathcal F}\mathcal B_t)(m)\lbrack h\rbrack\big)S(m+h)\Big\rangle.
 \end{align*}
 Since $\Phi$ and $\mathcal{B}_t$ are, by assumption, restrictions of continuous, affine linear mappings, the first two terms in the last expression vanish and we conclude
 \begin{align*}
 a_1^{(t)}(u,w)+c^{(t,m)}(u,w)=\Big\langle w,\big((\mathrm D_{\mathcal F}\mathcal B_t)(m)\lbrack h\rbrack\big)S(m)
 -\big((\mathrm D_{\mathcal F}\mathcal B_t)(m)\lbrack h\rbrack\big)S(m+h)\Big\rangle
 \end{align*}
 for all $w\in W$, i.e.,
 \begin{align*}
 \mathcal T_t(I_V+\mathcal{C}_{t,m}^V)u
 &=\big((\mathrm D_{\mathcal F}\mathcal B_t)(m)\lbrack h\rbrack\big)S(m)
 -\big((\mathrm D_{\mathcal F}\mathcal B_t)(m)\lbrack h\rbrack\big)S(m+h)\\
&=\big((\mathrm D_{\mathcal F}\mathcal B_t)(m)\lbrack h\rbrack\big)\big(S(m)-S(m+h)\big)
 \end{align*}
 thanks to part a) of Proposition \ref{well-posedness}.
 This yields
 \begin{align}\label{key}
 &\|S(m+h)-S(m)-(\mathrm D_{\mathcal F}S(m))\lbrack h\rbrack\|_V\\ \notag
 \leq&\|(I_V+\mathcal{C}_{t,m}^V)^{-1}\|_{\op}\cdot\|\mathcal{T}_t^{-1}\|_{\op}
 \cdot\|(\mathrm D_{\mathcal F}\mathcal B_t)(m)\|_{\op}\cdot\|h\|_X\cdot\|S(m)-S(m+h)\|_H\\ \notag
 \leq&\frac{\gamma}{c(t)}\|(I_V+\mathcal{C}_{t,m}^V)^{-1}\|_{\op}\cdot\|(\mathrm D_{\mathcal F}\mathcal B_t)(m)\|_{\Li(X,\Li(H,W^*))}\cdot\|h\|_X\cdot\|S(m)-S(m+h)\|_V.
 \end{align}
 In order to complete the proof, consider an arbitrary $m_0\in \mathcal G$ and any $\kappa\in(0,1)$.
By a simple continuity argument, we can find $\varrho'>0$ such that $B_{\varrho'}(m_0)\subseteq \mathcal G_t$, $\mathrm D_{\mathcal F}S$ is bounded on $B_{\varrho'}(m_0)$, and
\begin{align*}
\Lambda:=
\sup_{m\in B_{\varrho'}(m_0)}\frac{\gamma}{c(t)}\|(I_V+\mathcal{C}_{t,m}^V)^{-1}\|_{\op}\cdot\|(\mathrm D_{\mathcal F}\mathcal B_t)(m)\|_{\Li(X,\Li(H,W^*))}
<\infty.
\end{align*}
Now we choose $\varrho\in(0,\varrho')$ such that $2\Lambda\varrho<\kappa$.
For all $m_1,m_2\in B_{\varrho}(m_0)\subseteq B_{\varrho'}(m_0)$ we then derive, employing inequality \eqref{key} and the triangle inequality,
\begin{align*}
&\|S(m_1)-S(m_2)-\mathrm D_{\mathcal F}S(m_2)\lbrack m_1-m_2\rbrack\|_V\\
\leq&\Lambda\left(\|m_1-m_0\|_X+\|m_0-m_2\|_X\right)\|S(m_1)-S(m_2)\|_V
\leq\kappa\|S(m_1)-S(m_2)\|_V.
\end{align*}
Using the first estimate in \eqref{key}, we also have
\begin{align*}
\|S(m+h)-S(m)-(\mathrm D_{\mathcal F}S(m))\lbrack h\rbrack\|_H
&\leq\gamma\|S(m+h)-S(m)-(\mathrm D_{\mathcal F}S(m))\lbrack h\rbrack\|_V\\
&\leq\Lambda\cdot\|h\|_X\cdot\|S(m)-S(m+h)\|_H
\end{align*}
for $m\in \mathcal G_t$ and $h\in X\setminus\{0\}$ such that $m+h\in \mathcal G_t$.
The same line of argument as before finishes the proof.
 \end{proof}
 
  \vspace*{1ex}
 
\begin{remark}
Observe that the function $S$ in  Theorem \ref{tcc for general S} fulfils a very strong variant of the classical tangential cone condition
as the tangential cone constant $\kappa$ may be chosen arbitrarily small (of course, at the cost of choosing the radius $\varrho$ very small).
\end{remark}
 
 \vspace*{1ex}

In the context of Theorem \ref{tcc for general S} notice that a function given as a constant additive perturbation of a function satisfying a tangential cone condition fulfils the same tangential cone condition.

\subsection{Inverse problem with respect to the parameter $t$}
We assume in this subsection that $E$ is an open set of a Banach space $Y$. A similar line of argument as in the proof of Theorem \ref{regularity}
leads to the subsequent result.

\vspace*{1ex}  
 
 \begin{theorem}\label{regularity wrt t}
 Let $\nu\in\N\cup\{\infty\}$, $m\in U$ be fixed, and $\mathcal{O}_m$ a non-empty, open subset of $E$.
 Assume that problem \eqref{general 2} is strongly well-posed for all $t\in \mathcal{O}_m$. We further consider the mappings 
 \begin{align*}
 \pmb\Upphi:\mathcal{O}_m\rightarrow W^*;\,t\mapsto\upphi_t:= \pmb\Upphi(t)
 \end{align*}
 and
 \begin{align*}
 \mathcal{T}:E\rightarrow\Li(H,W^*);\,t\mapsto\mathcal{T}_t
 \end{align*}
 as well as the parameter-to-state operator
 \begin{align*}
 \pmb\uptau:\mathcal{O}_m\rightarrow V;\, t\mapsto u_{t},
 \end{align*}
 where $u_t=u_{t,m,\upphi_t}$ is the unique solution $u\in V$ of the problem
 \begin{align*}
  \forall\,w\in W:\,a_1^{(t)}(u,w)+c^{(t,m)}(u,w)=\upphi_t(w)=\langle w,\pmb\Upphi(t)\rangle.
 \end{align*}
 \begin{enumerate}
  \item If $\pmb\Upphi$, $\mathcal{T}$ and $\mathfrak{c}(\cdot,m)$ are $\nu$-times (continuously) Fr\'echet-differentiable on $\mathcal{O}_m$, then $\pmb\uptau$ is also $\nu$-times (continuously)
  Fr\'echet-differentiable on $\mathcal{O}_m$.
  \item If $\pmb\Upphi$, $\mathcal{T}$ and $\mathfrak{c}(\cdot,m)$ are analytic on $\mathcal{O}_m$, then $\pmb\uptau$ is also analytic on $\mathcal{O}_m$.
 \end{enumerate}
 \end{theorem}

\vspace*{1ex}  
 
 Assume that the hypotheses from Theorem \ref{regularity wrt t} hold.
We use a suitable variant of the representation \eqref{representation parameter-to-state operator} to calculate the Fr\'echet-derivative of
$\pmb\uptau$.
Hereafter, we shall give sufficient conditions that guarantee that $\pmb\uptau$ satisfies the tangential cone condition.
For that purpose, we first note that
\begin{align*}
\pmb\uptau(t)=\left(\mathcal{T}_t+\mathcal{B}_{t,m}j\right)^{-1}\upphi_t=\left(\mathcal{T}(t)+\mathcal{B}^{m}(t)j\right)^{-1}\pmb\Upphi(t),
\end{align*}
where
\begin{align*}
\mathcal{B}^m:E\rightarrow\Li(H,W^*);\,t\mapsto\mathcal{B}_{t,m}.
\end{align*}
Similarly as in the preceding subsection, we obtain
\begin{align*}
&\mathrm{D}_{\mathcal{F}}\pmb\uptau(t)\lbrack y\rbrack\\
=&-\left(\mathcal{T}(t)+\mathcal{B}^{m}(t)j\right)^{-1}
\left(\mathrm{D}_{\mathcal{F}}\mathcal{T}(t)\lbrack y\rbrack+\mathrm{D}_{\mathcal{F}}\mathcal{B}^{m}(t)\lbrack y\rbrack j\right)
\left(\mathcal{T}(t)+\mathcal{B}^{m}(t)j\right)^{-1}\pmb\Upphi(t)\\
&+\left(\mathcal{T}(t)+\mathcal{B}^{m}(t)j\right)^{-1}\mathrm{D}_{\mathcal{F}}\pmb\Upphi(t)\lbrack y\rbrack\\
=&-\left(\mathcal{T}(t)+\mathcal{B}^{m}(t)j\right)^{-1}
\left(\mathrm{D}_{\mathcal{F}}\mathcal{T}(t)\lbrack y\rbrack+\mathrm{D}_{\mathcal{F}}\mathcal{B}^{m}(t)\lbrack y\rbrack j\right)\pmb\uptau(t)\\
&+\left(\mathcal{T}(t)+\mathcal{B}^{m}(t)j\right)^{-1}\mathrm{D}_{\mathcal{F}}\pmb\Upphi(t)\lbrack y\rbrack
\end{align*}
for all $t\in\mathcal{O}_m$ and all $y\in Y$. So, $\mathrm{D}_{\mathcal{F}}\pmb\uptau(t)\lbrack y\rbrack$ is the unique element $u\in V$ such that
\begin{align}\label{derivative wrt t parameter-to-state operator}
a_1^{(t)}(u,w)+c^{(t,m)}(u,w)
=\Big\langle w,\mathrm D_{\mathcal F}\pmb\Upphi(t)\lbrack x\rbrack
-\left(\mathrm{D}_{\mathcal{F}}\mathcal{T}(t)\lbrack y\rbrack+\mathrm{D}_{\mathcal{F}}\mathcal{B}^{m}(t)\lbrack y\rbrack j\right)\pmb\uptau(t)\Big\rangle
\end{align}
for all $w\in W$.

\vspace*{1ex} 
 
\begin{theorem}\label{tcc for general tau}
 Let $m\in U$ be fixed and $\mathcal{O}_m$ a non-empty, open subset of $E$.
 Assume that problem \eqref{general 2} is strongly well-posed for all $t\in \mathcal{O}_m$.
We further consider a continuous affine linear mapping
 \begin{align*}
 \pmb\Upphi:Y\rightarrow W^*;\,t\mapsto\upphi_t
 \end{align*}
 as well as the parameter-to-state operator
 \begin{align*}
 \pmb\uptau:\mathcal{O}_m\rightarrow V;\, t\mapsto u_{t},
 \end{align*}
 where $u_t=u_{t,m,\upphi_t}$ is the unique solution $u\in V$ of the problem
 \begin{align*}
  \forall\,w\in W:\,a_1^{(t)}(u,w)+c^{(t,m)}(u,w)=\upphi_t(w)=\langle w,\pmb\Upphi(t)\rangle.
 \end{align*}
 Moreover, we assume that, for fixed $m$, $\mathfrak{c}(\cdot,m)$ and $\mathcal{T}$ are restrictions of continuous affine linear functions defined on $Y$
 and that the quantity $c(t)$ is chosen such that it depends continuously on $t$.
 In particular, $\pmb\uptau$ is continuously Fr\'echet-differentiable on $\mathcal{O}_m$ thanks to Theorem \ref{regularity wrt t}.
Then, for each $t_0\in \mathcal O_m$ and every $\kappa\in(0,1)$ there exists a constant $\varrho=\varrho(t_0,\kappa)>0$ such that $B_\varrho(t_0)\subseteq \mathcal O_m$,
 the Fr\'echet-derivative $\mathrm D_{\mathcal F}\pmb\uptau$ of $\pmb\uptau$ is bounded on $B_\varrho(t_0)$ and $\pmb\uptau$ satisfies on $B_\varrho(t_0)$ a
 $\kappa$-tangential cone condition w.r.t. both $\|\cdot\|_H$ and $\|\cdot\|_V$, i.e., we have
 \begin{align}\label{general tcc wrt t}
 \|\pmb\uptau(t_1)-\pmb\uptau(t_2)-(\mathrm D_{\mathcal F}\pmb\uptau(t_2))\lbrack t_1-t_2\rbrack\|_H\leq\kappa\|\pmb\uptau(t_1)-\pmb\uptau(t_2)\|_H
 \end{align}
 and
  \begin{align}\label{general tcc wrt t in V}
 \|\pmb\uptau(t_1)-\pmb\uptau(t_2)-(\mathrm D_{\mathcal F}\pmb\uptau(t_2))\lbrack t_1-t_2\rbrack\|_V\leq\kappa\|\pmb\uptau(t_1)-\pmb\uptau(t_2)\|_V
 \end{align}
 for all $t_1,t_2\in B_\varrho(t_0)$.
 \end{theorem}
 
 \vspace*{1ex} 
 
 \begin{proof}
  Let $t\in \mathcal O_m$, $h\in Y\setminus\{0\}$ such that $t+h\in \mathcal O_m$, let $w\in W$, and put
 $u:=\pmb\uptau(t+h)-\pmb\uptau(t)-(\mathrm D_{\mathcal F}\pmb\uptau(t))\lbrack h\rbrack$.
 Using \eqref{derivative wrt t parameter-to-state operator} and \eqref{****}, we infer
 \begin{align*}
&a_1^{(t)}(u,w)+c^{(t,m)}(u,w)\\
=&a_1^{(t)}(\pmb\uptau(t+h),w)+c^{(t,m)}(\pmb\uptau(t+h),w)\\
&-(a_1^{(t)}(\pmb\uptau(t),w)+c^{(t,m)}(\pmb\uptau(t),w))\\
&-(a_1^{(t)}((\mathrm D_{\mathcal F}\pmb\uptau(t))\lbrack h\rbrack,w)+c^{(t,m)}((\mathrm D_{\mathcal F}\pmb\uptau(t))\lbrack h\rbrack,w))\\
=&a_1^{(t+h)}(\pmb\uptau(t+h),w)+c^{(t+h,m)}(\pmb\uptau(t+h),w)
-\langle w,\pmb\Upphi(t)\rangle\\
&-\Big\langle w,\mathrm D_{\mathcal F}\pmb\Upphi(t)\lbrack h\rbrack
-\left(\mathrm{D}_{\mathcal{F}}\mathcal{T}(t)\lbrack h\rbrack+\mathrm{D}_{\mathcal{F}}\mathcal{B}^{m}(t)\lbrack h\rbrack j\right)\pmb\uptau(t)\Big\rangle\\
&+a_1^{(t)}(\pmb\uptau(t+h),w)-a_1^{(t+h)}(\pmb\uptau(t+h),w)\\
&+c^{(t,m)}(\pmb\uptau(t+h),w)-c^{(t+h,m)}(\pmb\uptau(t+h),w)\\
=&\langle w,\pmb\Upphi(t+h)-\pmb\Upphi(t)-\mathrm D_{\mathcal F}\pmb\Upphi(t)\rangle\\
&-\left(a_1^{(t+h)}(\pmb\uptau(t+h),w)-a_1^{(t)}(\pmb\uptau(t+h),w)
-\left\langle w,\mathrm{D}_{\mathcal{F}}\mathcal{T}(t)\lbrack h\rbrack\pmb\uptau(t+h)\right \rangle\right)\\
&+\left\langle w,\mathrm{D}_{\mathcal{F}}\mathcal{T}(t)\lbrack h\rbrack\pmb\uptau(t)-
\mathrm{D}_{\mathcal{F}}\mathcal{T}(t)\lbrack h\rbrack\pmb\uptau(t+h)\right \rangle\\
&-\left(c^{(t+h,m)}(\pmb\uptau(t+h),w)-c^{(t,m)}(\pmb\uptau(t+h),w)
-\Big\langle w,\mathrm{D}_{\mathcal{F}}\mathcal{B}^{m}(t)\lbrack h\rbrack j\pmb\uptau(t+h)\Big\rangle\right)\\
&+\Big\langle w,\mathrm{D}_{\mathcal{F}}\mathcal{B}^{m}(t)\lbrack h\rbrack j\pmb\uptau(t)
-\mathrm{D}_{\mathcal{F}}\mathcal{B}^{m}(t)\lbrack h\rbrack j\pmb\uptau(t+h)\Big\rangle.
 \end{align*}
 Since $\pmb\Upphi$, $\mathfrak{c}(\cdot,m)$ and $\mathcal{T}$ are restrictions of continuous affine linear functions, we deduce
 \begin{align*}
 a_1^{(t)}(u,w)+c^{(t,m)}(u,w)=\left\langle w,\Big(\mathrm{D}_{\mathcal{F}}\mathcal{T}(t)\lbrack h\rbrack+
 \mathrm{D}_{\mathcal{F}}\mathcal{B}^{m}(t)\lbrack h\rbrack j\Big)(\pmb\uptau(t)-\pmb\uptau(t+h))\right\rangle
 \end{align*}
 for all $w\in W$. Due to the well-posedness assumption, this yields
 \begin{align*}
 u=(I_V+\mathcal{C}_{t,m}^V)^{-1}\mathcal T_t^{-1}\Big(\mathrm{D}_{\mathcal{F}}\mathcal{T}(t)\lbrack h\rbrack+
 \mathrm{D}_{\mathcal{F}}\mathcal{B}^{m}(t)\lbrack h\rbrack\Big)(\pmb\uptau(t)-\pmb\uptau(t+h)),
 \end{align*}
 which implies
 \begin{align}\label{key 2}
 \|u\|_V
 &\leq\frac{1}{c(t)}\|(I_V+\mathcal{C}_{t,m}^V)^{-1}\|_{\op}\|\mathrm{D}_{\mathcal{F}}\mathcal{T}(t)+
 \mathrm{D}_{\mathcal{F}}\mathcal{B}^{m}(t)\|_{\op}\|h\|_X\|\pmb\uptau(t)-\pmb\uptau(t+h)\|_H\\ \notag
 &\leq\frac{\gamma}{c(t)}\|(I_V+\mathcal{C}_{t,m}^V)^{-1}\|_{\op}\|\mathrm{D}_{\mathcal{F}}\mathcal{T}(t)+
 \mathrm{D}_{\mathcal{F}}\mathcal{B}^{m}(t)\|_{\op}\|h\|_X\|\pmb\uptau(t)-\pmb\uptau(t+h)\|_V.
 \end{align}
Now, pick $m_0\in \mathcal G$ and fix $\kappa\in(0,1)$.
By a simple continuity argument, we can find $\varrho'>0$ such that $B_{\varrho'}(t_0)\subseteq \mathcal O_m$, $\mathrm D_{\mathcal F}\pmb\uptau$ is bounded on $B_{\varrho'}(t_0)$, and
\begin{align*}
\Uplambda:=
\sup_{t\in B_{\varrho'}(t_0)}\frac{\gamma}{c(t)}\|(I_V+\mathcal{C}_{t,m}^V)^{-1}\|_{\op}\|\mathrm{D}_{\mathcal{F}}\mathcal{T}(t)+
 \mathrm{D}_{\mathcal{F}}\mathcal{B}^{m}(t)\|_{\op}
<\infty.
\end{align*}
Now we choose $\varrho\in(0,\varrho')$ such that $2\Uplambda\varrho<\kappa$.
For all $t_1,t_2\in B_{\varrho}(t_0)\subseteq B_{\varrho'}(t_0)$ we then derive, employing inequality \eqref{key 2} and the triangle inequality,
\begin{align*}
&\|\pmb\uptau(t_1)-\pmb\uptau(t_2)-(\mathrm D_{\mathcal F}\pmb\uptau(t_2))\lbrack t_1-t_2\rbrack\|_V\\
\leq&\Uplambda\left(\|t_1-t_0\|_X+\|t_0-t_2\|_X\right)\|\pmb\uptau(t_1)-\pmb\uptau(t_2)\|_V
\leq\kappa\|\pmb\uptau(t_1)-\pmb\uptau(t_2)\|_V.
\end{align*}
The assertion for $\|\cdot\|_H$ instead of $\|\cdot\|_V$ is proved as in the proof of Theorem \ref{tcc for general S}.
 \end{proof}

\subsection{Inverse problem with respect to the parameter $(t,m)$}
We assume in this subsection once again that $E$ is an open set of a Banach space $Y$. We are thus dealing with a parameter-to-state map $\Uptheta : \mathcal{O} \subseteq E \times U \rightarrow V, \ (t,m) \mapsto \Uptheta(t,m)$ that depends on the two variables $m$ and $t$. Since we are interested in the Fr\'echet-differentiability of $\Uptheta$, it is worth to recall the following statement (see, e.g., \cite[VII.8.1 (b)]{AmannEscherII}):
 Let $X_j$ be normed spaces for $j\in\{1,2,3\}$, $U_k\subseteq X_k$ open and non-empty for $k\in\{1,2\}$, and $F:U_1\times U_2\rightarrow X_3$ a function with the following two properties.
 \begin{itemize}
  \item For every $x_2\in U_2$ the function
  \begin{align*}
  F^{x_2}:=F(\cdot,x_2):U_1\rightarrow X_3;\,x_1\mapsto F(x_1,x_2)
  \end{align*}
  is Fr\'echet-differentiable and the function
  \begin{align*}
  \mathrm{D}_{\mathcal{F}}^{(1)}F:U_1\times U_2\rightarrow\Li(X_1,X_3);\, (x_1,x_2)\mapsto (\mathrm{D}_{\mathcal{F}}F^{x_2})(x_1)
  \end{align*}
  is continuous.
    \item For every $x_1\in U_1$ the function
  \begin{align*}
  F_{x_1}:=F(x_1,\cdot):U_2\rightarrow X_3;\,x_2\mapsto F(x_1,x_2)
  \end{align*}
  is Fr\'echet-differentiable and the function
  \begin{align*}
  \mathrm{D}_{\mathcal{F}}^{(2)}F:U_1\times U_2\rightarrow\Li(X_2,X_3);\, (x_1,x_2)\mapsto (\mathrm{D}_{\mathcal{F}}F_{x_1})(x_2)
  \end{align*}
  is continuous.
 \end{itemize}
 Then $F$ itself is continuously Fr\'echet-differentiable with
 \begin{align*}
 (\mathrm{D}_{\mathcal{F}}F)(x_1,x_2)\lbrack\xi_1,\xi_2\rbrack=(\mathrm{D}_{\mathcal{F}}^{(1)}F)(x_1,x_2)\lbrack\xi_1\rbrack
 +(\mathrm{D}_{\mathcal{F}}^{(1)}F)(x_1,x_2)\lbrack\xi_2\rbrack
 \end{align*}
for all $(x_1,x_2)\in U_1\times U_2$ and $(\xi_1,\xi_2)\in X_1\times X_2$.

\vspace*{1ex}

In view of the previous findings, it is now clear how to prove the following theorem.

\vspace*{1ex}  
 
 \begin{theorem}\label{regularity wrt t and m}
 Let $\nu\in\N\cup\{\infty\}$ and let $\mathcal{O}$ be a non-empty, open subset of $E\times U$.
 Assume that problem \eqref{general 2} is strongly well-posed for all $(t,m)\in \mathcal{O}$. We further consider a mapping
 \begin{align*}
 \Uppsi:\mathcal{O}\rightarrow W^*;\,(t,m)\mapsto\uppsi_{t,m}
 \end{align*}
 as well as the parameter-to-state operator
 \begin{align*}
 \Uptheta:\mathcal{O}\rightarrow V;\, (t,m)\mapsto u_{t,m},
 \end{align*}
 where $u_t=u_{t,m,\uppsi_{t,m}}$ is the unique solution $u\in V$ of the problem
 \begin{align*}
  \forall\,w\in W:\,a_1^{(t)}(u,w)+c^{(t,m)}(u,w)=\uppsi_{t,m}(w)=\langle w,\Uppsi(t,m)\rangle.
 \end{align*}
 \begin{enumerate}
  \item If $\Uppsi$ and $\mathfrak{c}$ are both $\nu$-times (continuously) Fr\'echet-differentiable on $\mathcal{O}$, then $\Uptheta$ is also $\nu$-times (continuously)
  Fr\'echet-differentiable on $\mathcal{O}$.
  \item If $\Uppsi$ and $\mathfrak{c}$ are both analytic on $\mathcal{O}$, then $\Uptheta$ is also analytic on $\mathcal{O}$.
 \end{enumerate}
 \end{theorem}
 
 \vspace*{1ex}

\subsection{A specific situation suited for inverse scattering problems}
In this subsection we consider a special case, which encompasses in particular the inverse problem from THz tomography as considered in \cite{awts18} and the inverse medium problem treated in \cite{gbpl}. 
Throughout this subsection we make the general assumption that we are in the situation of Theorem \ref{global version} or Theorem \ref{local version}. 
However, we specify even more the situation considered there. 

\vspace*{1ex}

First, we fix $t\in E$ and we assume that $\lambda := \lambda(t)\not=0$. For that reason we shall not mention $t$ any more in this section and suppress it in our notation. 
Second, we assume that there is a non-empty, open set $\mathcal G_t=\mathcal G\subseteq U$
such that $\{t\}\times\mathcal G\subseteq\mathcal U$
resp. $\{t\}\times\mathcal G\subseteq\mathcal O$, depending whether we are in the situation of Theorem \ref{global version}
or Theorem \ref{local version}. Third, we consider a continuous and \emph{linear} function
\begin{align*}
\mathfrak{b}:X\rightarrow\mathcal{S}(H\times W,\mathbb K), \,m\mapsto b^{(m)}.
\end{align*}
Finally, let $a_3\in\mathcal{S}(H\times W,\mathbb K)$.
Note that in specific situations both $\mathfrak{b}$ and $a_3$ may (and indeed will in general) also depend on $t$ (see below),
but since $t$ is fixed, such a dependence plays no role in the following considerations.

\vspace*{1ex}

In what follows we suppose that $\mathfrak{a}_2$ is given by
\begin{align}\label{a2}
a_2^{(m)}(x,w)=-\frac{1}{\lambda}b^{(m)}(x,w) + a_3(x,w)
\end{align}
for $m\in \mathcal G$, $x\in H$ and $w\in W$.

It is important to observe the following: While $\mathfrak{b}$ and $a_3$ may also depend on $t$,
this is not allowed for $\mathfrak{a}_2$!
To put it another way, the dependencies of $\lambda$, $\mathfrak{b}$ and $a_3$
on $t$ must interact in such a way that $\mathfrak{a}_2$ does not depend on $t$ any more.

\vspace*{1ex}

\begin{remark}
 The above claim is fulfilled for the variational problems from THz tomography and the inverse medium problem: The fixed parameter $t$ corresponds to the frequency $k_0$ of the radiation. We further set $\lambda(t) = t^2$, $a_3(x,w) := (x|w)_{L^2(\Omega)}$, and $b^{(m)}(x,w) = t^2(mx|w)_{L^2(\Omega)}$ such that we obtain the variational formulation of \eqref{hhg}, \eqref{rbc}. Note that $a_1(x,w)$ represents the Robin-Laplace operator in this variational problem.
\end{remark}

\vspace*{1ex}

It is obvious that in this case $a_2^{(m)}\in\mathcal{S}(H\times W,\mathbb K)$.
Moreover, for $m,\widetilde m\in \mathcal G$ we calculate
\begin{align*}
\|\mathfrak{a}_2(m)-\mathfrak{a}_2(\widetilde m)\|_{\mathcal{S}(H\times W,\mathbb K)}
&= \sup_{\genfrac{}{}{0pt}{}{x\in H}{\|x\|_H\leq 1}}\sup_{\genfrac{}{}{0pt}{}{w\in W}{\|w\|_W\leq 1}}
\left|-\frac{1}{\lambda}b^{(m)}(x,w)+\frac{1}{\lambda}b^{(\widetilde m)}(x,w)\right|\\
&=\frac{1}{|\lambda|}\sup_{\genfrac{}{}{0pt}{}{x\in H}{\|x\|_H\leq 1}}\sup_{\genfrac{}{}{0pt}{}{w\in W}{\|w\|_W\leq 1}}
\left|b^{(m)}(x,w)-b^{(\widetilde m)}(x,w)\right|\\
&=\frac{1}{|\lambda|}\cdot\|\mathfrak{b}(m)-\mathfrak{b}(\widetilde m)\|_{\mathcal{S}(H\times W,\mathbb K)}
\xrightarrow[m\to\widetilde m]{}0.
\end{align*}
As a consequence, we see that $\mathfrak{a}_2$ is indeed continuous.
\\[0,2cm]
By the choice of $\mathcal G$ there exists for each $\varphi\in W^*$ and every $m\in \mathcal G$ a unique solution $u_{m,\varphi}\in V$
to problem \eqref{2}, i.e., a unique $u_{m,\varphi}\in V$ such that
 \begin{align}\label{2a}
 \forall\,w\in W:\,a_1(u_{m,\varphi},w)+\lambda a_2^{(m)}(u_{m,\varphi},w)=\varphi(w).
 \end{align}
 We now fix $v_0\in V$ and we put $\varphi_m:=b^{(m)}(v_0,\,\cdot\,)\in W^*$ for $m\in \mathcal G$.
 In the following our main objective is to examine the properties of the mapping
 \begin{align} \label{superposition_principle_abstract}
 S:\mathcal G\rightarrow V;\, m\mapsto u_m:=u_{m,\varphi_m}+v_0.
 \end{align}
 As an immediate consequence of Theorem \ref{tcc for general S} and \ref{tcc for general S} we arrive at the subsequent result.

\vspace*{1ex} 
 
 \begin{theorem}\label{tcc for S}
 The function $S$ is continuously Fr\'echet-differentiable. Moreover, for each $m_0\in \mathcal G$ and every $\kappa\in(0,1)$ there exists a constant $\varrho=\varrho(m_0,\kappa)>0$ such that $B_\varrho(m_0)\subseteq \mathcal G$,
 the Fr\'echet-derivative $\mathrm D_{\mathcal F}S$ of $S$ is bounded on $B_\varrho(m_0)$ and $S$ satisfies on $B_\varrho(m_0)$ a
 $\kappa$-tangential cone condition w.r.t. both $\|\cdot\|_H$ and $\|\cdot\|_V$, i.e., we have
 \begin{align}\label{tcc}
 \|S(m_1)-S(m_2)-(\mathrm D_{\mathcal F}S(m_2))\lbrack m_1-m_2\rbrack\|_H\leq\kappa\|S(m_1)-S(m_2)\|_H
 \end{align}
 and
  \begin{align}\label{tcc in V}
 \|S(m_1)-S(m_2)-(\mathrm D_{\mathcal F}S(m_2))\lbrack m_1-m_2\rbrack\|_V\leq\kappa\|S(m_1)-S(m_2)\|_V
 \end{align}
 for all $m_1,m_2\in B_\varrho(m_0)$.
 \end{theorem}
 
\vspace*{1ex}

\begin{remark}
 If we have uniqueness of a weak solution of the boundary value problem \eqref{hhg}, \eqref{rbc}, the above results directly yield the well-definedness of the respective parameter-to-state map $S$, its Fr\'echet-differentiability and the local validity of the tangential cone condition. The definition \eqref{superposition_principle_abstract} reflects the superposition principle, i.e., the function $v_0$ corresponds to the incident field $\ui$.
\end{remark}

 \section{Conclusion and outlook}
We have introduced an abstract, functional analytic framework based on form methods that seems to be
suited to the analysis of parameter identification problems arising from certain parameter-dependent, elliptic boundary value problems in divergence form,
which encompass equations that are of particular interest in the area of parameter identification, most notably the inverse medium problem and the inverse scattering problem of THz tomography.

Our main focus was on the question of the well-definedness and the analytic properties of the corresponding parameter-to-state operators.
The first and crucial step consisted in an operator theoretic reformulation of certain abstract variational problems, which provided an easy account to (global and local) well-posedness results, hence, to well-definedness results for the parameter-to-state operator.
In addition, it was this operator theoretic reformulation that allowed us to study the analytic properties of the parameter-to-state operator and to show that, under appropriate and reasonable conditions, this operator is Fr\'echet-differentiable, smooth, analytic, or fulfils are very strong version of the so-called tangential cone condition, which is often postulated for numerical solution techniques, but hard to verify.
In particular, our approach allows an insight into how the mathematical properties of the relevant inclusions, norms etc.~ influence the constant $\kappa$ that appears in the tangential cone condition. 
This is useful information when one chooses regularisation methods like, for instance, sequential subspace optimisation techniques, where $\kappa$ influences the algorithm.

In a follow-up paper, we apply our abstract results to a broad range of elliptic boundary value problems in divergence form with Dirichlet, Neumann, Robin, or mixed boundary conditions, including real world problems such as the inverse problem of THz tomography, thereby giving a far-reaching extension of previous results due to Bao and Li \cite{gbpl}.

 \bibliographystyle{siam}
\bibliography{bibliography_abstract_framework}

\pagestyle{myheadings}
\thispagestyle{plain}

\end{document}